\documentclass{amsart}

\usepackage{amssymb, amsmath, graphicx,  amsfonts, amsthm, url}
\usepackage{mathrsfs}

\begin{document}

\title[Geodesics and information on  equilibrium probabilities]{Geodesics and  dynamical information projections on the manifold of H\"older equilibrium probabilities}




\author{Artur O. Lopes}
\address{Inst. de Matematica e Estatistica - UFRGS - Porto Alegre - Brazil}
\curraddr{}
\email{arturoscar.lopes@gmail.com}

\author{Rafael O. Ruggiero}
\address{Dept. de Matematica - PUC - Rio de Janeiro - Brazil}
\curraddr{}
\email{rafael.o.ruggiero@gmail.com }



\date{}

\dedicatory{}

\begin{abstract}
We consider here the  discrete time  dynamics described by a transformation $T:M \to M$, where $T$ is either  the action of shift $T=\sigma$  on the symbolic space $M=\{1,2,...,d\}^\mathbb{N}$, or, $T$ describes  the action of a $d$ to $1$  expanding transformation $T:S^1 \to S^1$ of class $C^{1+\alpha}$ (\,for example   $x \to T(x) =d\, x $ (mod $1) $\,), where $M=S^1$ is the unit circle.
It is known that the infinite-dimensional manifold $\mathcal{N}$ of   equilibrium probabilities for H\"older potentials $A:M \to \mathbb{R}$ is an analytical manifold  and carries a natural Riemannian metric associated with the asymptotic variance. We show here that under the assumption of the existence of a Fourier-like Hilbert basis for the kernel of the Ruelle operator there exists geodesics paths. When $T=\sigma$ and $M=\{0,1\}^\mathbb{N}$ such basis exists.

In a different direction, we also consider
the KL-divergence $D_{KL}(\mu_1,\mu_2)$ for a  pair of equilibrium probabilities.   If  $D_{KL}(\mu_1,\mu_2)=0$, then $\mu_1=\mu_2$. Although $D_{KL}$  is not a metric in $\mathcal{N}$, it describes the proximity between $\mu_1$ and $\mu_2$. A natural problem is: for a fixed probability $\mu_1\in \mathcal{N}$  consider the probability $\mu_2$ in a certain set of probabilities in $\mathcal{N}$, which minimizes $D_{KL}(\mu_1,\mu_2)$. This minimization problem is a dynamical version of the main issues considered  in information projections.  We consider this problem in $\mathcal{N}$, a case where all probabilities are dynamically  invariant,  getting explicit equations for the solution sought. Triangle and Pythagorean inequalities  will be investigated.
\end{abstract}

\maketitle

\newtheorem{theorem}{Theorem}[section]
\newtheorem{lemma}[theorem]{Lemma}
\newtheorem{proposition}[theorem]{Proposition}
\newtheorem{corollary}[theorem]{Corollary}
\newtheorem{question}{Question}

\theoremstyle{definition}
\newtheorem{definition}[theorem]{Definition}
\newtheorem{remark}[theorem]{Remark}
\newtheorem{example}[theorem]{Example}

\newcommand{\fspace}[1]{\mathcal{#1}}
\newcommand{\qspace}[1]{\widehat{\mathcal{#1}}}
\newcommand{\spacem}[1]{\mathcal{#1}}
\newcommand{\op}[1]{\mathscr{#1}}
\newcommand{\dd}{\mathrm{d}}
\newcommand{\supp}{\operatorname{supp}}
\newcommand{\Var}{\operatorname{Var}}
\newcommand{\entropy}{\operatorname{h_{\fspace{X}}}}
\newcommand{\pressure}{\operatorname{Pr}}
\newcommand{\lspan}{\operatorname{span}}
\newcommand{\Id}{\operatorname{Id}}
\newcommand{\interior}{\operatorname{int}}
\newcommand{\rv}{\operatorname{rv}}
\newcommand{\Rot}{\operatorname{Rot}}
\newcommand{\Holder}{\operatorname{Hol}}
\newcommand{\eqdef}{\mathbin{\overset{\footnotesize{\mathrm{def}}}{=}}}

\newcommand{\cL}{\mathcal{L}}
\newcommand{\bR}{\mathbb{R}}
\newcommand{\bN}{{\mathbb N}}

\newcommand{\fr}{\partial}

\section{Introduction}

Recent developments about the analytic and geometric structure of the set of normalized potentials for expanding linear maps on the circle and the shift of finite symbols, reveal a rich, challenging context to explore classical problems of calculus of variations in infinite dimensional Riemannian manifolds. The metric we consider does not correspond (as explained in \cite{GKLM}) to the $2$-Wasserstein  metric on the space of probabilities (where probabilities have no dynamical content).

In the first part of the paper (Sections 1-4) we consider a time evolution on the space $\mathcal{N}$ of H\"{o}lder -equilibrium probabilities $\mu$, which can be parameterized  by H\"{o}lder  Jacobians $J_\mu:M \to (0,1)$ (see \cite{GKLM}). This provides the analytic structure on $\mathcal{N}$. We show the existence  of geodesics for a natural Riemannian metric on  $\mathcal{N}$ (previously introduced in \cite{GKLM}).  Given a probability $\mu\in \mathcal{N}$ and a tangent vector (a function)
$\varphi$, the Riemannian norm  $||\varphi||$ is described by the asymptotic variance of $\varphi$ with respect to $\mu$.
In this sense this metric is  naturally dynamically defined.
This Riemannian metric is
related (equal up to a constant value) to the one presented in \cite{McM} (also called the pressure metric in \cite{BCS}).

This point of view can be understood as a possible mathematical description of non-equilibrium Statistical Mechanics, where a continuous time evolution is observed in the space of probabilities (the geodesic flow). Given two H\"{o}lder  equilibrium states, one can ask about an {\it optimal} (in some natural and dynamical sense) path connecting these two probabilities; in this case, the path minimizing asymptotic variance of tangent vectors.
A related but different setting appears in \cite{Jiang}.

In the second part of the work (about Information Projections) we analyze issues related  to   geometry on the space of equilibrium measures. More precisely, the study of the minimization (or maximization) of the {\it distance} of a fixed probability $\mu_0$ to a given compact  set $K\subset \mathcal{N}$ (this set  is convex when parameterized by Jacobians, as explained in Section \ref{DIP}); the distance used (is not exactly a metric) is described by the relative entropy (also known as  Kullback-Leibler  divergence). We present analytic expression for critical points. Given a certain compact  set  $K$, we are interested in minimizing (or maximizing) relative entropy $\mu \in \mathcal{K}\to h(\mu_0,\mu)$ of $\mu\in \mathcal{K}$ with respect to $\mu_0$; this can be understood as a problem in Ergodic Optimization (see \cite{BLL})
with constraints, where the potential to be minimized (maximized) is the relative entropy  $\mu \to h(\mu_0,\mu)$ (see also the analytic expression \eqref{a1}).

Information projections   are  important  tools  in Deep Learning (see   \cite{Nielsen}, \cite{PolWu} or \cite{Good}), in the study of the Fisher Information (see \cite{Ama} and Section 5 in \cite{LR}), in the understanding of the maximum likelihood estimator and in Information Geometry,  where the probabilities on the associated  manifold do not have dynamical content (see \cite{Ama}).
 We quote F. Nielsen in \cite{Nielsen}:

{\it Information projections are a core concept of information sciences that are met whenever minimizing divergences.}

We consider such class of problems in a dynamical setting; in particular triangle and Pythagorean inequalities.

\smallskip

Now,  let us be more precise (in mathematical terms) about what we talked about above. We consider the  discrete time dynamics given by a transformation $T:M \to M$, where $T$ is either  the action of shift $T=\sigma$  on the symbolic space $M=\{1,2,...,d\}^\mathbb{N}$, or, $T$ describes  the action of a $d$ to $1$  expanding transformation $T:S^1 \to S^1$ of class $C^{1+\alpha}$ (\,for example   $x \to T(x) =d\, x $ (mod $1) $\,), where $M=S^1$ is the unit circle. For fixed $M$,
it is known that the set $\mathcal{N}$ of  equilibrium probabilities for H\"older potentials $A:M \to \mathbb{R}$ is an infinite dimensional, analytic manifold  and carries a natural Riemannian metric (see \cite{GKLM} and \cite{LR1}).  We say that a potential $A$ is normalized, if $\sum_{T(y)=x} e^{ A(y)}=1$, for all $x\in M$ (see \cite{PP} and \cite{GKLM}).  Points in $\mathcal{N}$ will be denoted indistinctly by normalized potentials $A$, or by  $\mu_A$, which is  the  equilibrium probability for the topological pressure $P(A)$ (this notation will be used in Subsection \ref{yu}). Equilibrium probabilities are sometimes called Gibbs probabilities (they are the same here).

According to \cite{GKLM}, given an equilibrium probability $\mu_A\in \mathcal{N}$, for the H\"older potential $A:M \longrightarrow \mathbb{R}$, the set of tangent vectors to
$\mathcal{N}$ at $\mu_A$ is the set of H\"older functions on the kernel of the Ruelle operator $\op{L}_A$ (see \eqref{K37} for definition). The Riemannian metric $g$ acting on tangent vectors at the base point $\mu_A$ is the $L^{2}$ inner product, $g_{A}(X,Y) = \int X\, Yd\mu_{A}$, where $A$ is a normalized
H\"older potential,  as  defined in \cite{GKLM}.



A study of the sectional curvatures of $\mathcal{N}$ is made in \cite{LR1}, where it is given a formula for the sectional curvatures in terms of an orthonormal basis of the tangent space at each point.
Equilibrium probabilities for potentials $A:\{0,1\}^\mathbb{N} \to \mathbb{R}$ that depend on the first two coordinates on the symbolic space  $M=\{0,1\}^\mathbb{N}$ are Markov probabilities on $\{0,1\}^\mathbb{N}$.
In this case, explicit examples  show that there exist pairs $(X,Y)$ on the tangent space to $\mathcal{N}$ where the absolute values of the sectional curvatures may attain arbitrarily large numbers, in contrast with finite dimensional Riemannian geometry. In  \cite{LR1} it is also shown that in this case the sectional curvature for pair of tangent vectors in $\mathcal{N}$ can be positive, zero, or negative. However, this two dimensional manifold has zero curvature at every point for the Riemannian structure inherited from $g$ (see \cite{LR1}).

It
is not known in the general case if the infinite-dimensional manifold $\mathcal{N}$ endowed with the Riemannian metric $g$ is complete. These facts strongly suggest that the
study of geodesics in $\mathcal{N}$ might be a subtle issue.


The purpose of the article is twofold. First of all, we deal with the problem of the existence of geodesics in $\mathcal{N}$ equipped with the $L^{2}$ Riemannian metric described in \cite{GKLM} and \cite{LR1}. This is the content of the first three sections.
The existence of geodesics in an infinite dimensional manifold is not a simple task.

\begin{definition} \label{suit} We say that the  equilibrium probability $\mu_A\in \mathcal{N}$ associated to the  H\"older potential $A$ is {\it Fourier-like}, if there exists a countable orthonormal Hilbert basis $\gamma_n$, $n \in \mathbb{N}$, of the kernel of the Ruelle operator $\op{L}_A$, and constants $\alpha>0, \beta>0$, such that,

I) the  functions $\gamma_n$, $n \in \mathbb{N}$, in   the family $ \mathcal{B} $ have $C^0$ and $L^2(\mu_A)$ norms uniformly bounded above by the constant $\beta>0$,

II) the  functions $\gamma_n$, $n \in \mathbb{N}$, in   the family $ \mathcal{B} $ have   $C^0$ and $L^2(\mu_A)$ norms uniformly bounded below by the constant $\alpha>0$.

We call such a basis a Fourier-like Hilbert basis (Fourier-like basis for short).

\end{definition}

The existence of a Fourier-like basis  for the kernel of the Ruelle operator plays an important role and its existence is discussed in the Appendix Subsection \ref{aqui1}.

One of our main result is:

\begin{theorem} \label{geodesic-potentials-1} Given
$M$, $T$ and a H\"older normalized potential  $A \in \mathcal{N}$, suppose there exist a Fourier-like Hilbert basis for the kernel of the Ruelle operator $\op{L}_A$. Then,  there exists an open ball $B_{r}(A)$ around $A$ such that for every $Q \in B_{r}(A)$ and every unit vector $X \in T_{B}\mathcal{N}$, there exists a unique geodesic $\gamma_{X} : (-\epsilon, \epsilon) \longrightarrow B_{r}(A)$ such that $\gamma_{X}(0) = Q$, $\gamma'_{X}(0) = X$, where $\epsilon >0$ depends on $A, Q$.

When
$M=\{0,1\}^\mathbb{N}$ and $T=\sigma$  we show the existence of a  Fourier-like Hilbert basis for the kernel of the Ruelle operator and then  it follows that geodesics exist as described above.
\end{theorem}

Subsection \ref{marbas1} shows the existence of an explicit  Fourier-type Hilbert basis for the kernel of the Ruelle operator in the case of Markov probabilities. The functions on this basis are constant in cylinder sets.
A result of independent interest is the existence of a Fourier-like basis for the
space $L^2 (\mu_A)$ which is the purpose of Subsection \ref{ger}.

In Section \ref{aqui2} we give the expression of the geodesic system of differential equation in some special coordinates $(r,s)\in(0,1)\times (0,1)$, for the two dimensional surface of Markov probabilities associated to two by two  row stochastic matrices
$$P=   \left(
\begin{array}{cc}
r & 1-r\\
1-s  & s
\end{array}\right).$$
We will exhibit two pictures showing geodesics paths on $(0,1)\times (0,1)$.

\medskip

Secondly, in Section \ref{DIP}  we deal with a different kind of calculus of variations problem in $\mathcal{N}$:  Information Projections for equilibrium probabilities. This problem is relevant in the context of Fisher information Theory, which holds for  probabilities that may not be invariant by any dynamical system. Our main object of study is the so-called KL-divergence, which is somehow considered a sort of distance between probabilities. We consider the KL-divergence for probabilities on $\mathcal{N}$ (which are all dynamically invariant). Let us comment briefly on some basic definitions and properties of this functional.

The notation in \cite{GKLM} was:  $\mu_A \in \mathcal{N}$  denotes the equilibrium probability associated with the normalized potential $A$. The function $J$, such that $J=e^A$, is called the Jacobian of the invariant probability $\mu_A$. Here, in Section \ref{DIP} it is more natural to use the notation: given the Jacobian $J$, we denote by $\mu_J$ the equilibrium probability for the potential  $A=\log J.$

The KL-divergence $D_{KL}(\mu_0,\mu_1)$ (also known as relative entropy $h(\mu,\mu_1)$) is defined for a pair of probabilities $\mu_0,\mu_1$.

Given two Jacobians $J_0$ and $J_1$ and the equilibrium probabilities $\mu_0:=\mu_{ J_0}\in \mathcal{N}$ and $\mu_1:=\mu_{J_1}\in \mathcal{N}$,
its Kullback-Leibler  divergence (or relative entropy) is given by
\begin{equation} \label{a1}
 D_{KL}(\mu_0\,|\, \mu_1) =\int (\log J_0 - \log J_1)\,\, d \mu_0\geq 0.\,\,\,
\end{equation}

 $D_{KL}$  is not a metric in the space of probabilities, however, it provides a measure of the proximity between $\mu_1$ and $\mu_2$.   If $D_{KL}(\mu_1,\mu_2)=0$, then $\mu_1=\mu_2$. A natural problem in information theory is the following: given a fixed probability $\mu_1$, to find the probability $\mu_2$ in a convex set of probabilities (not containing $\mu_1$) which minimizes $D_{KL}(\mu_1,\mu_2)$. This kind of minimization problem is one of the main issues in information projections.

A detailed study of the KL-divergence for equilibrium probabilities is described in \cite{LR}, \cite{LM2}, \cite{Cha1}, \cite{Cha2} and \cite{Cha}.

We analyze in Section \ref{DIP} in  the present article the information projection problem in the dynamical setting introduced in \cite{LR} based on Thermodynamics formalism. In this case, all probabilities are ergodic and they are all singular with respect to each other. In this setting, the basic tools of the calculus of Thermodynamics formalism as developed in \cite{GKLM} apply to the study of both the Riemannian geometry of $\mathcal{N}$, as shown in \cite{LR1} for instance, and to the study of the KL-divergence.

Moreover, the distance in $\mathcal{N}$ endowed with the $L^{2}$ metric and the calculus of variations of the KL-divergence, though quite different in nature, seem to be linked by the so-called  Pinsker inequality.

The Pinsker inequality (see \cite{Pre}) claims  that: if $p,q$ are two probabilities on a measurable space, then
$$ \delta(p,q)^2< \frac{1}{2} D_{KL}(p,q),$$
where $\delta(p,q)$ is the total variation distance.

On the other hand if  $p $ and $q $ are probability densities both supported on an interval $[0,1]$, then the Gy\"orfi inequality claims that
$$ D_{KL} (p,q) \leq \frac{1}{\inf_{x\in [0,1]} q(x)} ||p-q||_2^2.$$

So the KL-divergence is related with the $L^{2}$ distance between probabilities, and hence it is somehow related to the distance in $\mathcal{N}$. Therefore, it seems natural to us to try to investigate questions related to the minimization of the $D_{KL}$ divergence of equilibrium probabilities in parallel to the study of geodesics in $\mathcal{N}$. The second part of our paper can be considered as a first  attempt to tackle the subject.

Let us describe more precisely the main results concerning KL-divergence.

Denote by
$\Omega =\{1,2,...,d\}^\mathbb{N}$ the compact symbolic space with finite symbols.
The  Jacobian $J=e^A:\Omega \to(0,1)$ has the following properties: $J$ is a positive H\"older function such that
$\op{L}_{\log J}(1)=1,$ where $\op{L}_{\log J}$ is the Ruelle operator for the potential $\log J.$ To each Jacobian $J$ is associated  a unique shift   invariant probability $\mu=\mu_{\log J}$ (denoted  by $\mu_J$ for simplification, as mentioned before), such that,
$\op{L}_{\log J}^*(\mu_J)= \mu_J$, where  $\op{L}_{\log J}^*$ is the dual  of the Ruelle operator $\op{L}_{\log J}$. In our notation for Section \ref{DIP}
$\mu_J$ is the equilibrium probability for $\log J$.

Note that a convex combination of two distinct equilibrium probabilities $\mu_{J_0}$ and $\mu_{J_1}$ is not of the form $\mu_J$, for some  H\"older Jacobian $J$.

Given H\"older Jacobians $J_0$ and $ \tilde{J}_1 $, $\tilde{J}_1 \neq J_0$,
consider the  H\"older Jacobian $\mathfrak{J}_\lambda$, $\lambda\in [0,1]$, such that, $ \mathfrak{J}_\lambda = \lambda \tilde{J}_1  + (1- \lambda) J_0.$
In this case $ \mathfrak{J}_1=\tilde{J}_1$.

We denote by $\mu_{\mathfrak{J}_\lambda}$ the equilibrium probability for $\log \mathfrak{J}_\lambda$. The probability $\mu_0$ has Jacobian $J_0=\mathfrak{J}_0$ and $\mu_{\mathfrak{J}_1}$ has Jacobian $\tilde{J}_1=\mathfrak{J}_1$.
Given a fixed $J_1$ (corresponding to $\mu_1=\mu_{J_1}$), we are interested in estimating the derivatives of the Kullback-Liebler divergence (also known as relative entropy)
\smallskip

$\,\,\,\,\,\,\,\,\,\,\,\,\frac{d \,}{d \lambda} D_{KL}(\mu_{\mathfrak{J}_\lambda}, \mu_{J_1})|_{\lambda=0} =\frac{d}{d \lambda}[\,\int \log \mathfrak{J}_\lambda d \mu_{\mathfrak{J}_\lambda} - \int \log J_1 d
\mu_{\mathfrak{J}_\lambda}\,]\,|_{\lambda=0},$

and

$\,\,\,\,\,\,\,\,\,\,\,\,\frac{d \,}{d \lambda} D_{KL}(\mu_{J_1}, \mu_{\mathfrak{J}_\lambda})|_{\lambda=0} =\frac{d}{d \lambda}[\,\int \log J_1 d \mu_{J_1} - \int \log \mathfrak{J}_\lambda d
\mu_{J_1}\,]\,|_{\lambda=0}.$
\smallskip

This class of problems is related to the Pythagorean inequality
\smallskip

$\,\,\,\,\,\,\,\,\,\,\,\,\,\,\,\,\,\,\,\,\,\,\,\,D_{KL}(\mu_{\tilde{J}_1 }, \mu_{J_1}) \geq D_{KL}(\mu_{\tilde{J}_1 }, \mu_{J_0}) +  D_{KL}(\mu_{J_0}, \mu_{J_1}).$
\smallskip

One of our results in Section \ref{DIP}   is the computation:
\begin{proposition} \label{er23774}
\begin{equation} \label{a49} \frac{d}{d \lambda} D_{KL}(\mu_1, \mu_{\mathfrak{J}_\lambda})|_{\lambda=0}=\int (1-  \frac{\tilde{J}_1 }{ J_0}) d \mu_1.
\end{equation}
\end{proposition}

Given a {\bf convex set $\Theta_1$} of Jacobians $\tilde{J}$ and  $J_1\notin \Theta_1$, we consider the  related  problem:
find $\tilde{J}=J_0\in \Theta_1$  s. t. $D_{KL}(  \mu_{J_0}, \mu_{J_1})=\min_{\tilde{J}\in \Theta_1} D_{KL}(  \mu_{\tilde{J}}, \mu_{J_1}).$
We also consider:  find $\tilde{J}=J_0\in \Theta_1$ s. t. $D_{KL}(  \mu_{J_1}, \mu_{J_0})=\min_{\tilde{J}\in \Theta_1} D_{KL}(  \mu_{J_1}, \mu_{\tilde{J}}).$

The Second Law of Thermodynamics  corresponds to the case (see \cite{LR})
$$\frac{d \,}{d \lambda} D_{KL}(\mu_{\mathfrak{J}_\lambda}, \mu_{J_1})|_{\lambda=0}>0.$$

 This means that the relative entropy $ D_{KL}(\mu_{\mathfrak{J}_\lambda}, \mu_{J_1})$ is increasing infinitesimally  close to $\lambda=0.$ Note that we are not really in the realm of Thermodynamics of gases.

\medskip

Another issue is the study of  the probabilities $\mu^\lambda$ that are equilibrium for the family of potentials
\begin{equation} \label{a10102}  \lambda \log (\tilde{J}_1 ) + (1- \lambda) \log (J_0),
\end{equation}
$\lambda\in [0,1]$. We denote by $\mathfrak{J}^\lambda$   the Jacobian of the equilibrium probability $\mu^\lambda$ for the potential $\lambda \log (\tilde{J}_1 ) + (1- \lambda) \log (J_0)$ ($\mathfrak{J}_\lambda$ is different from $\mathfrak{J}^\lambda$).
The probability $\mu^1$ has  Jacobian $\tilde{J}_1=\mathfrak{J}^1 $ and the probability $\mu^0$ has Jacobian $J_0= \mathfrak{J}^0$.

We will also compute in Section \ref{DIP}:

\begin{proposition} Given $\mu_1=\mu_{J_1}$
\begin{equation} \label{chili0}\frac{d}{d \lambda}|_{\lambda=0} D_{K L} (\mu_1,\mu^\lambda) = - \int (\log \tilde{J}_1 - \log J_0) d \mu_1 + \int (\log \tilde{J}_1 - \log J_0) \,d \mu^0.
\end{equation}

The inequality $0\leq \frac{d}{d \lambda}|_{\lambda=0}   D_{K L}      (\mu_1,\mu^\lambda) $ is  related to the Pythagorean inequality:

 $$ D_{K L}      (\mu_1,\mu^0)  +  D_{K L}      (\mu^0,\mu_{ \tilde{J}_1})    \leq    D_{K L}      (\mu_1,\mu_{ \tilde{J}_1}). $$
\end{proposition}

We also describe what is  the  dynamical Bregman divergence for two probabilities in $\mathcal{N}$ (see expression \eqref{equio3}).


\section{Preliminaries for the study of geodesics in $\mathcal{N}$}

\subsection{Basics of Riemannian Geometry}

Let us start by introducing some basic notions of Riemannian geometry. Given an infinite dimensional  $C^{\infty}$ manifold $(\mathfrak{M},g)$ equipped with a smooth Riemannian metric $g$, let $T\,\mathfrak{M}$ be the tangent bundle and
$T_{1} \mathfrak{M}$ be the set of unit norm tangent vectors of $(\mathfrak{M},g)$,  known as  the unit tangent bundle. Let $\chi(\mathfrak{M})$ be the set of $C^{\infty}$ vector fields of $\mathfrak{M}$.

Given a smooth function $f :\mathcal{N} \longrightarrow \mathbb{R}$, the derivative of $f$ with respect to a vector field $X \in \chi (\mathcal{N} )$ will be denoted by $X(f)$.
The Lie bracket of two vector fields $X, Y \in \chi(\mathcal{N} )$ is the vector field whose action on the set of functions $f: \mathcal{N}  \longrightarrow \mathbb{R}$ is
given by $[X,Y](f) = X(Y(f)) - Y(X(f))$.

The \textit{Levi-Civita connection} of $(\mathcal{N} ,g)$, $\nabla : \chi(\mathcal{N} )\times \chi(\mathcal{N} ) \longrightarrow \chi(\mathcal{N} )$, with notation $\nabla(X,Y) = \nabla_{X}Y$, is the affine operator
characterized by the following properties:
\begin{enumerate}
\item Compatibility with the metric $g$:
$$ Xg(Y,Z) = g(\nabla_{X}Y, Z) + g(Y, \nabla_{X}Z) $$
for every triple of vector fields $X, Y, Z$.
\item Absence of torsion: $$ \nabla_{X}Y - \nabla_{Y}X = [X,Y].$$
\item For every smooth scalar function $f$ and vector fields $X,Y \in \chi(\mathcal{N} )$ we have
\begin{itemize}
\item $ \nabla_{fX}Y = f\nabla_{X}Y$,
\item Leibniz rule: $ \nabla_{X}(fY) = X(f)Y + f\nabla_{X}Y$.
\end{itemize}
\end{enumerate}

The expression of $\nabla_{X}Y$ can be obtained explicitly from the expression of the Riemannian metric, in dual form. Namely, given two vector fields $X, Y \in \chi(\mathcal{N} )$,
and $Z \in \chi(\mathcal{N} )$ we have
\begin{eqnarray*}
g(\nabla_{X}Y, Z) & = & \frac{1}{2}(Xg(Y,Z) + Yg(Z, X) -Zg(X,Y) \\
& - & g([X,Z], Y) -g([Y,Z],X) -g([X,Y], Z)) .
\end{eqnarray*}

A smooth curve $\gamma(t) \subset \mathcal{N}$, for $t$ in an interval $I \subset \mathbb{R}$, is called a \textit{geodesic} if it satisfies
$$ \nabla_{\gamma'(t) } \gamma'(t) = 0 $$
for every $t \in I$. The properties of the Levi-Civita connection imply that geodesics have constant speed (see Subsection 2.7), so we can restrict ourselves to $T_{1}\mathcal{N}$ to study geodesics. In finite dimensional Riemannian manifolds, geodesics are solutions of a system of second order differential equations in the manifold. This follows from taking coordinates and writing explicitly the geodesic condition in terms of the coordinate vector fields. For infinite dimensional Riemannian manifolds, a more analytic approach is needed. For Riemannian manifolds which are complete as metric spaces, the so-called Palais-Smale method is often applied to prove the existence of geodesics (see \cite{Kli} for instance). We do not know if the manifold $\mathcal{N}$ is complete  when endowed with the $L^{2}$ Riemannian metric. So we shall adopt an alternative method to deal with the existence of geodesics based strongly on the analytic properties of $\mathcal{N}$.

\subsection{Preliminaries of the analytic structure of the set of normalized potentials} \label{Prel}

\smallskip

We recall for the reader the basic results that we will need later following
the content of the  first sections of \cite{LR1}.

\begin{definition}  Let $ (X, | .|)$ and  $(Y, |.|)$  Banach spaces and $V$ an open subset of $ X.$
Given $k\in  \mathbb{N}$, a function $F : V\to Y$ is called $k$-differentiable in $x $, if for each $j=1, ..., k$,
there exists a $j$-linear bounded transformation
$$D^j F(x) : \underbrace{X \times X \times ... \times X}_j  \to Y,$$
such that,
$$D^{j -1}F(x + v_j )(v_1, ..., v_{j-1}) \,-\, D^{j-1}F(x)(v_1, ..., v_{j-1}) = D^jF(x)(v_1, ..., v_j ) + o_ j (v_j ),\,\,$$
where
$$\,\, o_j
: X \to Y, \,\,\text{ satisfies,}\,\,  \lim_{v\to 0}
\frac{|o_j (v)|_Y}{
|v|_X
}= 0
$$

By definition $F$ has derivatives of all orders in $V$, if for any  $x\in V$ and any $k\in  \mathbb{N}$, the function $F$ is
$k$-differentiable in $x$.

\end{definition}

\begin{definition} Let $X, Y$ be Banach  spaces and $V$ an open subset of $X$. A function
$F : V \to X$ is called analytic on $V$ when $F$ has derivatives of all orders in $V$, and for each
$x \in V$ there exists an open neighborhood $V_x$ of $x$ in $V$, such that, for all $v\in  V_x$, we have that
$$
F(x + v) \,-\, F(x) = \sum_{j=1}^\infty\,
\frac{1}{n !}\,\,
D^j F(x)v^j,$$
where
$D^j F(x)v^j = D^j F(x)(v, . . . , v) $ and $ D_j F(x) $ is the $j $-th derivative of $F$ in $x$.

\end{definition}

Above we use the notation of section 3.2 in \cite{SiSS}.

\medskip

$\mathcal{N}$ can be expressed locally in coordinates via analytic charts (see \cite{GKLM}).

\subsection{Fundamental formulae from  Thermodynamic Formalism} \label{yu}

\noindent

\smallskip

For a fixed $\alpha>0$ we denote by $\text{Hol}$ the set of $\alpha$-H\"older functions on $M$.
For a H\"older potential $B: M \to \mathbb{R}$ in $\text{Hol}$ we define the Ruelle operator (sometimes called transfer operator) - which acts on H\"older functions $f: M \longrightarrow \mathbb{R}$ -   by the law
\begin{equation} \label{K37}f \to  \op{L}_B f(x) = \sum_{T(y) = x} e^{B(y)} f(y).\end{equation}

Given a potential $B \in \text{Hol} $ and  the associated Ruelle operator $\op{L}_B$, consider
the corresponding main eigenvalue $\lambda_{B}$ and eigenfunction $h_B$ (see \cite{PP} for the proof of their existence).
 As mentioned before,  $\mu_B$ denotes the equilibrium probability for  the topological pressure $P(B)$.
We say that the potential $B$ is normalized if  $ \op{L}_B (1)=1.$ When $B$ is normalized the eigenvalue is $1$ and the eigenfunction is equal to $1$.  In this case, using the notation of \cite{GKLM} (as mentioned before) we get  $\op{L}_B^* (\mu_B)=\mu_B$. In the present section is more natural and didactic to use the notation of \cite{GKLM}.

The function
\begin{equation} \label{K20} \Pi (B) = B + \log(h_{B}) - \log(h_{B}(T)) -\log(\lambda_{B}) \end{equation}
describes  the projection of the space of potentials $B$ on $\text{Hol}$  onto the analytic manifold of normalized potentials $\mathcal{N}$.

The potential $\Pi (B)$ is normalized.

We identify below $T_A \mathcal{N}$ with the affine subspace $\{A + X\,:\, X \in T_A \mathcal{N}\}.$

The function $\Pi$ is analytic  on $B$ (see \cite{PP} or \cite{GKLM}) and therefore has first and second derivatives.  Given the potential $B$, then the map $D_B \Pi : T_{B}\mathcal{N} \longrightarrow T_{\Pi(B)}\mathcal{N} $ given by
$$ D_B \Pi (X) = \frac{\partial}{\partial t}(\Pi(B+ tX))|_{t=0} $$
should be considered as a linear map from Hol to itself (with the H\"{o}lder norm on Hol). Moreover, the second derivative $D^2_B \Pi$ should be interpreted as a bilinear form from Hol $\times$ Hol to Hol, and is given by

$$ D^2_B \Pi(X,Y)  = \frac{\partial^2}{\partial t \partial s}(\Pi(B+ tX +sY))|_{t=s=0}. $$

We denote by $||A||_\alpha$ the $\alpha$-H\"{o}lder norm of an $\alpha$-H\"{o}lder function $A$.

We would like to study the geometry of the
projection $\Pi$ restricted to the tangent space $T_{A}\mathcal{N}$ into the manifold $\mathcal{N}$
(namely, to get bounds for its first and second derivatives with respect to the potential viewed as a variable) for a given normalized potential $A$.

For an H\"older normalized potential $A$
the space $T_{A}\mathcal{N}$ is a linear subspace of functions (the set of H\"older functions on the kernel of the Ruelle operator $\op{L}_A$) and the derivative map $D\, \Pi$ is analytic when restricted to it.

We denote by $E_0 =E_0^A$ the set of H\"{o}lder functions $g$,  such that $\int g d \mu_A =0.$ Note that $E_0^A$ is contained in $T_{A} \mathcal{N}.$

{The claims of the next Lemma are  taken from \cite{LR1}} and they are based mainly on results of \cite{GKLM} (see also  \cite{SiSS}, \cite{BCV}).

\begin{lemma} \label{derivative-bounds0}
Let $\Lambda : \operatorname{Hol} \longrightarrow \mathbb{R}$,
$H : \operatorname{Hol}\longrightarrow \operatorname{Hol}$ be given, respectively,  by $ \Lambda (B) = \lambda_{B},$ $ H(B) = h_{B}$.
Then we have
\begin{enumerate}
\item The maps $\Lambda$, $H$, and $A \longrightarrow \mu_{A}$  are analytic.
\item For a normalized $B$ we get that $D_{B}\log(\Lambda) (\psi) = \int \psi d\mu_{B},$
\item $ D^{2}_{B}\log(\Lambda) (\eta, \psi) = \int \eta \psi d\mu_{B},$
where $\psi , \eta $ are at $T_{B}\mathcal{N}$.
\item For any H\"{o}lder potential $A$ we have
$$D_{A}H(X) = h_{A} \int (\,[\,(I - \op{L}_{T,A}|_{E_0^A})^{-1}\,( 1- h_{A}) \,].\,X)\, d\mu_{A} .$$
If $A$ is normalized, we have $D_{A}H =0$,
\item If $A$ is a normalized potential, then for every function $X \in T_{A}\mathcal{N}$ we have
\begin{itemize}
\item $\int X d\mu_{A} =0$.
\item $D_{A}\Pi(X) = X$.
\end{itemize}

\end{enumerate}
\end{lemma}


The law that takes an H\"older potential $B$ to its normalization $A=\Pi (B)$ is differentiable according to section 2.2 in \cite{GKLM}.

As a consequence of the analytic properties of the functions $\Lambda, H$ we have the following:

\begin{proposition} \label{dpi}
Given a normalized potential $A \in \mathcal{N}$ and $\delta>0$ there exists $r >0$, such that, for every H\"older continuous function
$B$ in the ball $B_{r}(A)$ of radius $r$ around $A$, the norms of $D_{B}\Pi$ and $D^{2}_{B}\Pi$
restricted to the functions in $T_{A}\mathcal{N}$ satisfy
$$ \parallel (D_{B}\Pi)\mid_{T_{A}\mathcal{N}} - I \parallel \leq  \delta$$
$$ \parallel (D^{2}_{B}\Pi)\mid_{T_{A}\mathcal{N}} + I \parallel \leq \delta.$$

\end{proposition}

 In the above for linear operators we use the operator norm (in Hol we consider the sup norm) and for bilinear forms, we use also the sup norm (see section 2.3 in \cite{GKLM}).

\medskip

\subsection{On the Calculus of Thermodynamical formalism}

The following result proved in \cite{LR1} describes a  formula to calculate derivatives of integrals of vector fields. This rule will be important to estimate the coefficients of the first fundamental form of the Riemannian metric in $\mathcal{N}$ in order to deal with the problem of the existence of geodesics.

\begin{lemma} \label{Leibniz}
Let $A \in \mathcal{N}$ and let $\gamma: (-\epsilon, \epsilon) \longrightarrow \mathcal{N}$ be a smooth curve such that $\gamma(0) =A$. Let $X(t) = \gamma'(t)$, and let $Y$ be a smooth vector field tangent to $\mathcal{N}$ defined in an open neighborhood of $A$. Denote by $Y(t)= Y(\gamma(t))$. Then the derivative of $\int Y(t) d\mu_{\gamma(t)}$ with respect to the parameter $t$ is
$$ 	\frac{d}{dt} \int Y(t)d\mu_{\gamma(t)} = \int \frac{dY(t)}{dt}d\mu_{\gamma(t)} + \int Y(t)X(t)d\mu_{\gamma(t)} $$
for every $ t\in (-\epsilon, \epsilon)$.
\end{lemma}

\section{The existence of geodesics in $\mathcal{N}$} \label{geo}
\medskip
Since the manifold of normalized potentials is an infinite dimensional manifold,
the usual way of proving the existence of geodesics via solutions of  ordinary differential equations with coefficients in the set of Cristoffel
symbols  does  not follow right away.

When
$M=\{0,1\}^\mathbb{N}$ and $T=\sigma$  we will show the existence of a  Fourier-like Hilbert basis for the kernel of the Ruelle operator and then  it follows that geodesics exists (see subsection \ref{aqui1}).
In the general case, Theorem \ref{geodesic-potentials-1} express in more precise terms the main result we will get.

It is not clear that the Palais-Smale theory works in our case.  However, what we shall show is in some
sense a weak Palais-Smale condition for our Riemannian manifold: roughly speaking, we shall construct a sequence of approximated solutions
of the Euler-Lagrange equation having as a limit a true solution of the equation.

We would like to point out  that we will not use any of the classical results on Hilbert manifolds.

We shall develop a strategy to prove the existence of geodesics based on the fact that there exist a (countable) complete orthogonal set $\varphi_n$,
$n \in \mathbb{N}$, on $\mathcal{L}^2 (\mu_A)$ according to Theorem 3.5 in \cite{KS} (see also \cite{CHLS}). Taking an order for the basis, and subspaces $\sigma_{m}$ generated by the first $m$ vectors of the basis, we shall study the system of differential equations of geodesics restricted to the submanifolds
obtained by $\Pi$-projections of open sets of the subspaces $\sigma_{m}$ in the manifold $\mathcal{N}$.
We shall be more precise in the forthcoming subsections.
\medskip

\subsection{Good Coordinate systems for the manifold of normalized potentials}

\begin{lemma} \label{basis}
Let $A$ be normalized potential, and let $B_{r}(A)$ is the open  neighborhood of $A$ in $\mathcal{N}$ given in Proposition \ref{dpi}).
Let $e_{n}$ be an orthonormal basis of $T_{A}\mathcal{N}$. Then we have,
\begin{enumerate}
\item Let $Q \in \Pi^{-1}(B_{r}(A))$, and let $\bar{e}_{n}$ be an extension of $e_{n}$ in the plane
$T_{A}\mathcal{N}$ as a constant vector field. Then, the functions
$$v_{n}(\Pi(Q)) = D_{Q}\Pi(\bar{e}_{n})$$
form a basis for $T_{\Pi(Q)}\mathcal{N}$ and
$$  \mid \langle v_{n}(\Pi(Q)), v_{m}(\Pi(Q)) \rangle  - \delta_{nm} \mid \leq \delta,$$
where $\delta_{nm}$ is the Kronecker function : $\delta_{nm} = 1$ if $n=m$, and $0$ otherwise.
\item There exists $b>0$, such that, the map $\Pi$ restricted to the sets
$$
U_{m} (b)= \{ \sum_{i=1}^{m} t_{i} e_{i}, \mbox{ } \mid t_{i} \mid < b\}$$
is an embedding into a $m$-dimensional submanifold $S_{m} \subset \mathcal{N}$, for every $m \in \mathbb{N}$.
\end{enumerate}
\end{lemma}

\begin{proof}

From Proposition \ref{dpi}, we know that $D_{A} \Pi\mid_{T_{A}\mathcal{A}} = I$ and that $D_{Q}\Pi \mid_{T_{A}\mathcal{A}}$ is close to the identity
if $B=\Pi(Q) \in B_{r}(A)$. Hence, if we chose $Q = A + \sum_{i=1}^{m} t_{i}w_{i}$ in a way that $\parallel B- A \parallel < r$ then the vectors $ v_{n}(\Pi(Q)) = D_{Q}\Pi(e_{n})$
will be almost perpendicular at $T_{B}\mathcal{N}$. This yields that the vectors $v_{n}(B)$ are linearly independent in $T_{B}\mathcal{N}$ and therefore,
the map $\Pi$ has constant rank $m$ in $U_{m}$. By the local form of immersions, the image $S_{m} = \Pi(U_{m})$ is an analytic submanifold of $\mathcal{N}$
of dimension $m$.
\end{proof}

\subsection{A system of partial differential equations for geodesic vector fields}

A natural way to show that geodesics exist in $\mathcal{N}$ is to show that geodesics exist in each analytic submanifold $S_{m}$ ( of dimension $m$) and then take the limit as $m$ goes to $+\infty$. On each submanifold $S_{m}$, a system $\Sigma_{m}$ of partial differential equations will arise from the restriction of the system of differential equations of geodesics. Our strategy to solve an initial value problem for the geodesic equation is to solve the initial value problem for $\Sigma_{m}$ in each submanifold $S_{m}$, then take the limit of the sequence $\gamma_{m}$ of solutions as $m \rightarrow +\infty$, and finally, we have to show that the limit gives rise to a geodesic of $\mathcal{N}$ solving the initial value problem.

The existence of a limit solution depends on uniform estimates of the coefficients of the systems $\Sigma_{m}$. So the main goal of this subsection is to obtain an explicit expression of the geodesic systems $\Sigma_{m}$ in terms of the coordinates in $S_{m}$, and show that their coefficients have uniformly bounded norms in an open neighborhood of each normalized potential. Proposition \ref{dpi} will be crucial for this purpose.

To get the expressions of the systems $\Sigma_{m}$, we apply the ideas of the finite dimensional case. So let $A \in S_{m}$, $v \in T_{A}S_{m}$, and suppose that the solution of the system $\Sigma_{m}$, $\gamma_{m}(t)$, given by the initial conditions
$\gamma_{m}(0) = A$, $\gamma'_{m}(0)= v$ exists. We shall characterize $\gamma_{m}$ in terms of a differential equation in the submanifold $S_{m}$ that has a unique solution.
We would like to point out that the differential equations of geodesics in the finite dimensional case are written in terms of the Christoffel coefficients. However, we shall avoid the use of Christoffel coefficients and obtain a simpler, equivalent system for the geodesics, of partial differential equations of first order.

Let $X(t) = \gamma'(t)$, since it is geodesic, $\nabla_{X}X =0$, where $\nabla$ is the Levi-Civita
connection of the Riemannian metric in $\mathcal{N}$. This implies that
\begin{equation} \label{K61}\langle \nabla_{X}X, Y\rangle  =0, \end{equation}
for every $Y \in T_{\gamma(t)}\mathcal{N}$. By the expression of the Levi-Civita connection in terms of the metric (see the end of Section \ref{Prel}), we have
\begin{equation} \label{K22}\langle \nabla_{X}X, Y\rangle  = X\langle X,Y\rangle  - \frac{1}{2}Y\langle X,X\rangle  -\langle X,[X,Y]\rangle , \end{equation}
where $X(f)$ means the derivative of a scalar function $f$ with respect to $X$.

In particular, the energy of geodesics is constant,
\begin{equation}  \frac{1}{2}X\langle X, X \rangle = \langle \nabla_{X}X, X\rangle  = 0.\end{equation}

So let us restrict ourselves to the energy level of vector field $X$ with constant norm equal to 1. In this case, the equation of geodesics  and the expression of the Levi-Civita connection in terms of the metric   gives
$$ 0 = \langle \nabla_{X}X, Y\rangle  = X\langle X,Y\rangle  -\langle X,[X,Y]\rangle ,$$
or equivalently,
\begin{equation} \label{K64}X\langle X,Y\rangle   = \langle X,[X,Y]\rangle, \end{equation}
for every vector field $Y$.

 Let $e_{i}$ for $i=1, 2, .., m$ be the orthonormal vector fields in $T_{A}S_{m}$ given in Proposition \ref{basis}, let $\Phi : U_{m} \longrightarrow S_{m} $ be given by
$$\Phi(t_{1}, t_{2},..,t_{m}) = \Pi (\sum_{i=1}^{m} t_{i}e_{i})$$
that is a coordinate system defined in an open neighborhood $U_{m}$ of $0 \in T_{A}S_{m}$,  whose image is the smooth $m$-dimensional submanifold $S_{m}$.

Let $X_{n} = D\Phi(e_{n})$ be the coordinate vector fields tangent to $S_{m}$. Replacing in the expression of the geodesic equation above we have

$$ X\langle X,X_{n}\rangle  = \langle X,[X,X_{n}]\rangle . $$
This set of equations might be
used to show the existence of the geodesic vector field. Let us write down the system explicitly.

 Let $ X = \sum_{i=1}^{m} x_{i} X_{i}$, and let $\bar{x}_{i} = \langle X,X_{i}\rangle$. The differential equation of the geodesic vector field $X$
is equivalent to

$$ X\langle X,X_{n}\rangle  = \langle X,[X,X_{n}]\rangle  = \langle X, [\sum_{i=1}^{m} x_{i} X_{i},X_{n}]\rangle ,$$
and we observe that

$$ [\sum_{i=1}^{m} x_{i} X_{i},X_{n}] = \sum_{i=1}^{m} [x_{i}X_{i},X_{n}] = \sum_{i=1}^{m} (x_{i}[X_{i},X_{n}] - X_{n}(x_{i})X_{i}) ,$$
and since the vector fields $X_{n}$ commute, we finally get

$$ [\sum_{i=1}^{m} x_{i} X_{i},X_{n}] = \sum_{i=1}^{m} - X_{n}(x_{i})X_{i}.  $$

Hence we can write the differential equation for $X$ as

\begin{eqnarray*}
X(\bar{x}_{n}) = X\langle X,X_{n}\rangle  & = & -\langle X,\sum_{i=1}^{m}  X_{n}(x_{i})X_{i}\rangle \\
&  = & - \sum_{i=1}^{m}   \langle X,X_{n}(x_{i})X_{i}\rangle  = - \sum_{i=1}^{m}   X_{n}(x_{i})\bar{x}_{i}.
\end{eqnarray*}

In terms of $\frac{d}{dt}$, $\frac{d}{dt_{n}}$ we obtain a system $\Sigma_{m}$ of first order partial differential equations

\begin{eqnarray} \label{ney}
\Sigma_{m} := \mbox{ } \frac{d}{dt}(\bar{x}_{n})  & = & - \sum_{i=1}^{m}  \frac{d}{dt_{n}}(x_{i})\bar{x}_{i}. , \mbox{ } n=1,2,..,m.
\end{eqnarray}

The above system of differential equations gives rise to a system of partial differential equations for the functions $\bar{x}_{i}$.
Indeed, let $X=(x_{1},x_{2},..,x_{m})$, $\bar{X} = (\bar{x}_{1}, \bar{x}_{2},..,\bar{x}_{m})$, and let $M_{m}$ be the matrix of the first fundamental form in the basis $v_{i}$, namely,
$$ (M_{m})_{ij} = \langle X_{i}, X_{j} \rangle .$$
Then we have that $\bar{X} = M_{m}X$, and replacing this identity in the  initial system  (\ref{ney}) we get a system of first order,
quasi-linear partial differential equations (see chapter 7 in \cite{BJS} for definition and properties)  for the functions $x_{i}$ whose coefficients depend on the entries of the matrices $(M_{m})^{-1}$ and $\frac{d}{dt_{n}}((M_{m})^{-1})$: let $(M_{m})^{-1}_i$ be the $i$-th row of the matrix $(M_{m})^{-1}$. Then we have
$$\frac{d}{dt}(\bar{x}_{n}) = - \sum_{i=1}^{m}  \frac{d}{dt_{n}}(<(M_{m})^{-1}_i , \bar{X}>)\bar{x}_{i}. , \mbox{ } n=1,2,..,m, $$
\medskip
where $<(M_{m})^{-1}_i , \bar{X}>$ is the Euclidian inner product of the $i$-th row $(M_{m})^{-1}_i $ and the vector $X$.

\textbf{Remark}: Actually, the Christoffel coefficients of the Riemannian metric involve the derivatives of the entries of the first fundamental form of the metric. So it is not surprising that such derivatives appear in any formulation of the problem of the existence of geodesics.

\subsection{Uniform bounds for the PDE geodesic systems in a neighborhood of a Fourier-like probability}

In this subsection, we shall estimate the sup norm of the coefficients of the system of partial differential equations obtained in the previous section, in a neighborhood of a normalized potential corresponding to a Fourier-like equilibrium measure for the shift of two symbols. The main result is the following:

\begin{proposition} \label{system-bound}
Let $A \in \mathcal{N}$ be the normalized potential associated to a equilibrium probability of two symbols. There exists an open neighborhood $B_{r}(A) \subset \mathcal{N}$ and $D>0$ such that the coefficients of the
quasilinear systems of partial differential equations
$$\frac{d}{dt}(\bar{x}_{n}) = - \sum_{i=1}^{m}  \frac{d}{dt_{n}}(<(M_{m})^{-1}_i , \bar{X}>)\bar{x}_{i}. , \mbox{ } n=1,2,..,m $$
are uniformly bounded above by $D$.
\end{proposition}

Recall that a quasilinear system of partial differential equations of vector functions $x_{i}(t_{1}, t_{2},.., t_{m}) \in \mathbb{R}$ is a system of the form
$$ F(t_{i}, x_{j}, \frac{dx_{j}}{dt_{i}}) = 0 $$
where $F$ is a quadratic function of the variables  $x_{j}, \frac{dx_{j}}{dt_{i}}$. The system in Proposition \ref{system-bound} is a particular case, resembling the usual system of differential equations for geodesics
obtained by using the Christoffel coefficients.

A family of probabilities that are Fourier-like is given by the following Lemma:

\begin{lemma} \label{maxent-basis}
Let $A \in \mathcal{N}$ be the normalized  H\"older potential associated to an equilibrium probability $\mu$ on $M=\{0,1\}^\mathbb{N}$. Then, there exist $\alpha,\beta>0$, and an orthonormal basis of $T_{A}\mathcal{N}$ given by  continuous functions $\{e_{n}\}$, such that, the supremum of $e_{n}$ is $C^0$ and  $L^2(\mu)$ bounded above by $\beta$, and below by $\alpha$, for every $n$.
\end{lemma}

For the proof see Appendix  Section \ref{marbas}.
\smallskip

The estimates for the coefficients of the systems rely in a crucial way on the following result:

\begin{corollary} \label{basis-local-bound}
Let $A \in \mathcal{N}$ be the normalized potential associated to a Fourier-like equilibrium probability. Denote by $e_n$, the associated basis satisfying the conditions I) and II) of Definition  \ref{suit}. Let $\bar{e}_{n}$ be the extension of $e_{n}$ in the plane $T_{A}\mathcal{N}$ as a constant vector field. Then, there exists an open neighborhood $U \subset \mathcal{N}$ containing $A$, and $\rho >0$ such that
\begin{enumerate}
\item For every $B=\Pi(s_{1},s_{2},..,s_{m}) \in \Pi(U)$, the family of functions
$$\{ X_{n}(B) = D_{(s_{1},s_{2},..,s_{m})}\Pi(\bar{e}_{n}) \}$$ is a basis for  $T_{B}\mathcal{N}$.
\item The sup norm of each element of the basis $X_{n}(B)$ is bounded above by $\rho$.
\end{enumerate}
\end{corollary}

\begin{proof}

The corollary follows from Lemma \ref{maxent-basis} and Proposition \ref{dpi}.
\end{proof}

Let us consider the norm for matrices $\parallel M \parallel_{sup} = sup \{\mid M_{ij} \mid  \} $.

\begin{lemma} \label{bounded-coeff}
Let $A \in \mathcal{N}$ be the normalized potential associated to a Fourier-like equilibrium probability $\mu_A$. Then, there exists $C>0$ such that the norms of the matrices $(M_{m})^{-1}$, and the coefficients of $\frac{d}{dt_{n}}((M_{m})^{-1})$ are uniformly bounded by $C$ in the neighborhood $B_{r}(A)$.
\end{lemma}

\begin{proof}

The coefficients of the first fundamental form $M_{m}$ at a point $B \in B_{r}(A)$ are
$$\langle X_{i}(B),X_{j}(B)\rangle  = \int X_{i}(B)X_{j}(B) d\mu_{B} .$$

By Lemma \ref{basis} and Lemma \ref{dpi}, the matrix $M_{m}$ is a perturbation of the identity at every point $B \in B_{r}(A)$. This yields that the matrix
$(M_{m})^{-1}$ is close to the identity and its norm is uniformly (in $m$) bounded above in $B_{r}(A)$.

As for the derivative $\frac{d}{dt_{n}}((M_{m})^{-1}) =( M_{m})^{-1}\frac{d}{dt_{n}}(M_{m})(M_{m})^{-1}$, we notice that at the point $A$ we have $M_{m}= I_{m}$, the $m \times m$ identity matrix, and the coefficients of $\frac{d}{dt_{n}}(M_{m})$ are the derivatives of the terms $\langle X_{i},X_{j}\rangle $.  According to Lemma \ref{Leibniz} we have
$$ \frac{d}{dt_{n}}\langle X_{i},X_{j}\rangle  = \int (\frac{d}{dt_{n}}(X_{i}) X_{j} + X_{i} \frac{d}{dt_{n}}(X_{j}) + X_{i}X_{j}X_{n})d\mu_{B}. $$
Let us estimate the sup norms of each of these terms at a point $B \in B_{r}(A)$. First observe that $B = \Pi(s_{1},s_{2},..,s_{m})$ for
some vector $(s_{1},s_{2},..,s_{m} ) $ close to $(0,0,..0)$. Then we have

$$\frac{d}{dt_{n}}(X_{i}(B) ) = \frac{d}{dt_{n}}(D_{(s_{1},s_{2},..,s_{m} )}\Pi (e_{i}) )= \frac{d^{2}}{dt_{n}dt_{i}}(\Pi((s_{1},s_{2},..,s_{m} ) + t_{i} e_{i}+ t_{n}e_{n})) .$$
The sup norm of such a term is bounded above by $1+ \delta$ according to Proposition \ref{dpi}, therefore, the sup norms of the integrals $\int \frac{d}{dt_{n}}(X_{i}) X_{j}d\mu_{B}$ and $ \int X_{i} \frac{d}{dt_{n}}(X_{j} )d\mu_{B}$ are bounded above by $1+ \delta$.

 Moreover, the term $\int X_{i}X_{j}X_{n}d\mu_{B}$ satisfies

$$ \mid \int X_{i}X_{j}X_{n}d\mu_{B} \mid \leq \mid X_{i}(B)X_{j}(B)X_{n}(B) \mid_{\infty} , $$
and by Corollary \ref{basis-local-bound}, we have that $\mid X_{i}(B)X_{j}(B)X_{n}(B) \mid_{\infty} \leq (\rho)^{3}$, where $\rho $ is the upper bound for the elements of the basis in Corollary \ref{basis-local-bound}. Joining the above estimates we get that the coefficients of the first fundamental form $M_{m}$ are bounded above by $2(1+\delta) + (\rho)^{3}$ for every $B \in B_{r}(A)$. Since the matrices $M_{m}$ are uniformly close to the identity, the matrices $\frac{d}{dt_{n}}((M_{m})^{-1})$ are uniformly close to
$\frac{d}{dt_{n}}(M_{m})$ in $B_{r}(A)$ thus proving the lemma.

\end{proof}

The proof of Proposition \ref{system-bound} follows from Corollary \ref{basis-local-bound} and Lemma \ref{bounded-coeff}.
\medskip

\medskip

\subsection{First order systems of ordinary differential equations equivalent to first order PDE's}

Let us start this subsection with some standard basic results of the theory of first order partial differential equations. We follow {\bf Chapter 3 in} the book by L. C. Evans \cite{Evans}, but the subject is quite well known and there are many other classical references.

Let $F: \mathbb{R}^{n} \times \mathbb{R} \times \bar{U} : \longrightarrow \mathbb{R} $ be a $C^{2}$ function where $U$ is an open subset of $\mathbb{R}^{n}$ and $\bar{U}$ is its closure.
The system of first order, partial differential equations defined by $F$ is given by
$$ F(Du, u, x) =0$$
where $u: \bar{U} \longrightarrow \mathbb{R}$ is the unknown. Let us write
$$F(p,z,x) = F(p_{1}, p_{2},..,p_{n}, z, x_{1}, x_{2},..,x_{n})$$ and denote by
$$ D_{p}F = (F_{p_{1}}, F_{p_{2}}, ..,F_{p_{n}}), \mbox{ } D_{z}F = F_{z}, \mbox{ } D_{x}F = (F_{x_{1}}, F_{x_{2}}, ..,D_{x_{n}}) $$
the differentials of $F$ with respect to the variables $p, z, x$. The theory of the characteristics associates a system of first order differential equations to the system $F(Du,u,x)=0$ in the following way. We look for smooth curves $x(s) = (x^{1}(s),..,x^{n}(s))$
for $s \in \mathcal{I}$ defined in some open interval, and consider the function $z(s) = u(x(s))$. Let $p(s) = Du(x(s))$, where $p(s) = (p^{1}(s),..,p^{n}(s))$ is given by $p^{i}(s) = u_{x_{i}}(x(s))$. Differentiating with respect to $s$ we obtain the characteristic equations
\begin{eqnarray*}
p'(s) & = & -D_{x}F(p(s),z(s),x(s)) - D_{z}F(P(s),z(s),x(s))p(s) \\
z'(s) & = & D_{p}F(p(s),z(s),x(s))p(s) \\
x'(s) & = & D_{p}F(p(s),z(s),x(s))
\end{eqnarray*}

This setting extends of course to smooth finite dimensional manifolds, by taking local coordinate systems.

Euler-Lagrange equations in a Riemannian manifold, a system of second order differential equations, is equivalent to a first order system of partial differential equations in the tangent bundle of the manifold. The above procedure applied to this system gives rise to the Hamilton equations in the cotangent bundle, a system of ordinary first order differential equations.

Euler-Lagrange equations in the case of Riemannian metrics are expressed in terms of the Levi-Civita connection by the system
$$ \langle \nabla_{X}X, X_{i} \rangle = 0$$
where $X$ is the vector field tangent to a geodesic and $X_{i}$, $i=1,2,..,n$ is a coordinate basis of the tangent space of the $n$-dimensional manifold. This is exactly what we did in the previous subsection for each submanifold $S_{m}$. The tangent space $T\mathcal{N}$ and the cotangent space $T^{*}\mathcal{N}$ of $\mathcal{N}$ are analytic manifolds as well, and we are looking for solutions of Euler-Lagrange equations in finite dimensional submanifolds of $T\mathcal{N}$.

Therefore, as a consequence of Lemma \ref{bounded-coeff} and Theorem \ref{Picard} in the last section, we get a result on the existence of solutions for the partial differential equation of geodesics under the Fourier-like condition.

\begin{lemma} \label{geod-exist}
Let $A \in \mathcal{N}$ be the normalized H\"older potential associated to the equilibrium probability $\mu$ on $M=\{0,1\}^\mathbb{N}$.
Then there exist $\rho >0$, $D>0$, such that given a unit vector $X(0) \in T_{A}\mathcal{N}$ there exists
a unique analytic curve $\gamma:(-\rho, \rho ) \longrightarrow \mathcal{N}$ such that $\gamma(0) =A$, and $\gamma'(t) =X(t)$ is the unique solution of the equation (\ref{ney}) whose initial condition is $X(0)$. The solution $X(t)$ is defined in an interval $\mid t \mid \leq \rho$,
and the norms of $X(t)$, $X'(t)$ are bounded by $D$ for every $\mid t \mid \leq \rho$. An analogous result holds for every $Q \in B_{r}(A)$, where $r>0$ is given in Proposition \ref{system-bound}.
\end{lemma}

\begin{proof}
Let us show the statement for $A$, the statement for $Q \in B_{r}(A)$ follows from the same arguments. By the theory of first order partial differential equations, the system (\ref{ney}) that is a second order, partial differential system in the curve $\gamma(t)$
is equivalent to a system of first order ordinary differential
equations $\frac{d}{dt}Y = F_{m}(Y)$ where the function $F_{m}$ depends on the first fundamental form $A$ and its derivatives with respect to the
coordinates $t_{n}$. These functions have uniformly bounded norm in the neighborhood $B(r)$ and are analytic. Then, Theorem \ref{Picard}
implies the existence and uniqueness of
solutions of the ordinary differential equations, namely, there exists $\rho>0$ such that the solution $\gamma_{m}(t)$ of (1) with initial condition
$\gamma_{m}(0) = A$, $\gamma_{m}'(0) = X(0)$, is unique and defined in $(-\rho, \rho)$.
$$ \frac{d}{dt} \parallel Y \parallel \leq \parallel F \parallel \parallel Y \parallel $$
which yields that

The uniform bound for the sup norm of $F_{m}$ in $B(r)$ implies that there exists $\rho>0$ such that
the analytic solutions $\gamma_{m}(t)$ are defined in $(-\rho, \rho)$ and are uniformly bounded in this interval.

Then Theorem \ref{Arzela-2} implies that there exists a convergent subsequence with limit $\gamma(t)$ analytic in the interval $(-\rho, \rho)$.
The function $\gamma(t)$ is tangent to the curve of vectors $X(t)$ that s the limit of the convergent subsequence of the curves $\gamma_{m}'(t) = X_{m}(t)$ in $(-\rho, \rho)$.
\bigskip

\textbf{Claim:} The curve $\gamma(t)$ is a geodesic.
\bigskip

Since $X_{m}(t)$ converges uniformly to $X(t)$ in the interval $(-\rho, \rho)$ we have that given $\epsilon >0$ there exists $m_{\epsilon}$ such that for every $m \geq m_{\epsilon}$
we have
$$\parallel F_{m}(X_{m}'(t) ) - F_{m}(X(t)) \parallel_{\infty} \leq k\parallel X_{m}(t)) - X(t) \parallel_{\infty} \leq \epsilon$$
where $k$ is a constant depending on the (uniform) bounds of the first derivatives of the functions $F_{m}$. So we get that $X(t)$ is an
approximate solution of the systems defined by the functions $F_{m}$:
\begin{eqnarray*}
\parallel X' - F_{m}(X) \parallel_{\infty} & \leq & \parallel X' - X_{m}' \parallel_{\infty}+ \parallel X_{m}' - F_{m}(X_{m}) \parallel_{\infty }+ \parallel F_{m}(X_{m} - F_{m}(X) \parallel_{\infty} \\
& \leq & 2\epsilon
\end{eqnarray*}
if we choose $m_{\epsilon}$ such that $\parallel X_{m}'- X ' \parallel_{\infty} <\epsilon$ for every $m \geq m_{\epsilon}$ as well. Now, notice that
the equation $\frac{d}{dt} Y = F_{m}(Y)$ is equivalent to the system $\langle \nabla_{Y} Y,v_{k} \rangle =0$, for every $0<k \leq m$, which means that
$$ \mid \langle \nabla_{X} X, v_{k} \rangle \mid \leq \epsilon$$
for every $ 0< k \leq m$. Since $\epsilon$ may be chosen arbitrarily, we conclude that $\langle \nabla_{X} X, v_{m} \rangle =0$ for every $m$, which implies that
the vector field $\nabla_{X}X$ is identically zero, because the collection of the vectors $v_{m}$ is a base for the $L^{2}$ inner product in $T\mathcal{N}$.
This yields that the curve $\gamma(t)$ is a geodesic as we claimed.

\end{proof}

\subsection{ On the existence and uniqueness of solutions of differential equations in $\mathcal{N}$}
\bigskip

Let us now proceed to the proof of Picard's Theorem in our infinite dimensional setting. We start with the Arzela-Ascoli theorem. We shall state the main results for the shift and we claim that for the case of expanding maps $T(x) = 2x (mod. 1)$ in $S^{1}$ the results one can get are analogous.

\begin{theorem} \label{Arzela-Ascoli}
Let $(X,d)$ be a second countable compact metric space (namely, there exists a countable dense subset).
Let $\mathcal{F}$ be a family of functions $f : X \longrightarrow \mathbb{R}$ that is uniformly bounded and equicontinuous. Then
every sequence in $\mathcal{F}$ has a convergent subsequence in the set of continuous functions.
\end{theorem}

\begin{proof}
The proof follows from the same steps of the usual version of the theorem for compact subsets of $\mathbb{R}^{n}$.
\end{proof}

The above implies:

\begin{lemma} \label{shift-Arzela}
Let $\Sigma = \{0, 1\}^{\mathbb{N}}$, endowed with the metric
$$d(\{ a_{n}\}, \{b_{n}\}) = \frac{1}{2}\sum_{i=0}^{\infty} \frac{\mid a_{i} - b_{i} \mid}{2^{i}}.$$
Let $\operatorname{Hol}_{C,\alpha}(\Sigma)$ be the set of H\"{o}lder continuous functions $f : \Sigma \longrightarrow \mathbb{R}$ with constant $C$ and exponent $\alpha$ endowed
with the sup norm. Then, every subset of $\operatorname{Hol}_{C, \alpha}$ of uniformly bounded functions is precompact.
\end{lemma}

\begin{proof}
First of all, observe that $(\Sigma,d)$ is a compact metric space with a countable dense subset, the set of periodic sequences of $0$'s and $1$'s. Then Theorem \ref{Arzela-Ascoli}
holds, and since the set of functions in $\operatorname{Hol}_{C,\alpha}$ is equicontinuous, every uniformly bounded subset has a convergent subsequence.
\end{proof}

Next, let us study the precompactness of the set of analytic curves of normalized potentials
$\gamma: (a,b) \longrightarrow \operatorname{Hol}_{C,\alpha}(X)$ endowed with the sup norm.
By analytic we mean that $\gamma(t)$ depends analytically on the parameter $t \in (a,b)$.

\begin{proposition} \label{Arzela-2}
Let $\Gamma_{C,\alpha}([a,b],\Sigma)$ be the set of curves $\gamma : [a,b] \longrightarrow \operatorname{Hol}_{C,\alpha}(\Sigma)$ of normalized potentials
which are analytic in $(a , b)$ and continuous in $[a,b]$, endowed with the sup norm. Then every family of functions in $\Gamma_{C.\alpha}([a,b],\Sigma)$
that is uniformly bounded and equicontinuous has a convergent subsequence. Namely, there exists a continuous function $\gamma_{\infty} :[a,b] \longrightarrow
\operatorname{Hol}_{C,\alpha}(\Sigma)$ that is analytic on $(a,b)$ and a sequence of functions in $\Gamma_{C.\alpha}([a,b],\Sigma)$ converging uniformly to $\gamma_{\infty}$.
\end{proposition}

\begin{proof}
Let $\gamma_{n} \in \Gamma_{C,\alpha}([a,b],\Sigma)$ be a sequence of uniformly bounded curves. For simplicity, let us suppose that $a = -r, b= r$ for some $0 <r \leq 1$, and let us center the series expansion at $t_{0}=0$ (for different center of expansion the argument is just analogous).
This implies that we get an expression in power series for each $\gamma_{n}(t)$ of the form
$$ \gamma_{n}(t) = \sum_{m=0}^{\infty} a_{m}^{n}(p) t^{m} $$
where $a_{m}^{n}: \Sigma \longrightarrow \mathbb{R}$ are functions in $\operatorname{Hol}_{C,\alpha}(\Sigma)$. Since the functions $\gamma_{n}$ are uniformly
bounded by a constant $L>0$ in $(-r, r)$, we have that $\parallel a_{0}^{n} \parallel_{\infty} \leq L$ for every $n$ and by Lemma \ref{shift-Arzela} there exists
a convergent subsequence $a_{o}^{n^{0}_{i}}$ whose limit is a function $A_{0}$. Since the radius of convergence of all the series is $r$, we have that
$\limsup_{n} (\mid a_{m}^{n}(p) \mid )^{\frac{1}{n}} = \frac{1}{r}$ and therefore
$$ \parallel a_{m}^{n} \parallel_{\infty} \leq \frac{1}{r^{m}}$$
for every $n, m$. So the family of functions $\mathcal{F}_{m} = \{ a_{m}^{n}\}$ is uniformly bounded and we can apply again Lemma \ref{shift-Arzela}.
So there exists a subsequence $n^{0}_{n^{1}_{j}}$ of the indices $n^{0}_{j}$ such that the functions $a_{m}^{n^{0}_{j}}$ converge to a function
$A_{1} \in \operatorname{Hol}_{C,\alpha}(\Sigma)$. By induction, we get a subsequence $\gamma_{N_{k}}$ of the functions $\gamma_{n}$ such that
the first $k+1$ coefficients of their series expansions converge to functions $A_{0}, A_{1}, .., A_{k}$ in $\operatorname{Hol}_{C,\alpha}(\Sigma)$.

Consider the function
$$ \gamma_{\infty}(t) = \sum_{m=0}^{\infty} A_{m}(t) . $$
By the choice of the $A_{m}$'s, the above series converges with the same convergence radius of the functions $\gamma_{n}$. Moreover, it is easy to check that
$\gamma_{\infty}(t)$ is a curve of functions in $\operatorname{Hol}_{C,\alpha}(\Sigma)$, and we have that the sequence $\gamma_{N_{k}}$ converges uniformly
on compact sets to $\gamma_{\infty}$ in the sup norm. Indeed, let $[a,b] \subset (-r,r)$, since the functions $\gamma_{n}$ are uniformly bounded given
$\epsilon >0$ there exists $m_{\epsilon}>0$ such that for every $n \in \mathbb{N}$, $k \geq m_{\epsilon}$ we have
$$ \mid \sum_{k}^{\infty} a_{k}^{n}(p) t ^{k}\mid \leq \epsilon $$
for every $ p \in \Sigma$. The same holds for the series $\gamma_{\infty}$. This yields
\begin{eqnarray*}
\parallel \gamma_{\infty}(t) -\gamma_{n}(t) \parallel_{\infty} & \leq & \sum_{m=0}^{m_{\epsilon}} \parallel A_{m} - a_{m}^{N_{k}} \parallel_{\infty} t^{m} +
\parallel \sum_{m_{\epsilon}+1}^{\infty} (A_{m} - a_{m}^{N_{k}})\parallel_{\infty} t^{m}\\
& \leq & \sum_{m=0}^{m_{\epsilon}} \parallel A_{m} - a_{m}^{n} \parallel_{\infty} t^{m} + 2\epsilon.
\end{eqnarray*}
Since the functions $a_{m}^{N_{k}}$ converge uniformly to the function $A_{m}$, we can chose $k$ large enough such that $\parallel (A_{m} - a_{m}^{N_{k}})\parallel_{\infty} \leq \frac{\epsilon}{m}$, and therefore
$$\parallel \gamma_{\infty}(t) -\gamma_{n}(t) \parallel_{\infty} \leq 3\epsilon, $$
for every $t \in [-r,r]$, and since $\epsilon$ can be chosen arbitrarily we get the lemma.
\end{proof}

Now, we can state Picard's Theorem for differential equations in $\mathcal{N}$.

\begin{theorem} \label{Picard}
Let $ F : [x,y] \times U \longrightarrow \operatorname{Hol}_{C,\alpha}(\Sigma)$ be an analytic function in $t \in (x,y)$ and in $\operatorname{Hol}_{C,\alpha}(\Sigma)$,
where $U$ is an open subset of $(\operatorname{Hol}_{C,\alpha}(\Sigma))^{n}$.
Then, given $(t_{0}, f_{1}, f_{2}, ..,f_{n}) \in (x,y) \times U$ there exists a unique solution of the differential equation $\frac{d}{dt}X(t) = F(t,X(t))$
defined in a certain interval $X: (t_{0}-\epsilon,t_{0} + \epsilon) \longrightarrow U$ that is analytic and satisfies $X(t_{0}) = (f_{1},f_{2},..,f_{n})$.
\end{theorem}

\begin{proof}

The proof mimics the usual proof of Picard's theorem applying the idea of contraction operators. The operator
$$ L(g )(t) = (f_{1}, f_{2}, ..,f_{n}) + \int_{t_{0}}^{t} F(s,g(s)) ds $$
is defined in the set of continuous curves $ g : [x,y] \longrightarrow (\operatorname{Hol}_{C,\alpha}(\Sigma))^{n}$ that are analytic on $(x,y)$.
According to Lemma \ref{Arzela-2}, this set of curves endowed with the sup norm is co-compact. Now, as in the proof of the usual version of
Picard's theorem, there exists a small interval $( t_{0}- \epsilon, t_{0}+\epsilon) $, where $\epsilon >0$ depends on the sup norm of the first derivatives of the function $F$,
where the above operator restricted to curves defined in $( t_{0}- \epsilon, t_{0}+\epsilon) $ is a contraction. Therefore, by Lemma \ref{Arzela-2}, there
exists a unique fixed point $X(t)$ that must be the solution of the equation claimed in the statement. The solution is analytic since the function $F$ is analytic.
\end{proof}

\subsection{Geodesic accessibility of the set of potentials associated to Fourier-like equilibrium measures on symbolic spaces with two symbols}

The purpose of the subsection is to show:

\begin{theorem} \label{complete}
Let $A \in \mathcal{N}$ be the potential associated with a equilibrium probability of the shift of two symbols, and let $B \in \mathcal{N}$. Then, there exists a geodesic of $\mathcal{N}$ endowed with the variance Riemannian metric joining the two points.
\end{theorem}

The idea of the proof is inspired by the Palais-Smale condition: we shall construct a sequence of analytic curves joining two points whose lengths converge to the distance $d(A,B)$ in the Riemannian metric. Then, we show that the sequence has a convergent subsequence, in the set of analytic curves joining the two points, to a curve $\gamma :[0,1] \longrightarrow \mathcal{N}$, and by the general theory of geodesics this curve is a critical point of the length and thus a solution of the equation $\nabla_{\gamma'(t)}\gamma'(t) =0$.

We start by considering the curve $c(t) = A(1-t) + tB$, $c: [0,1] \longrightarrow \text{Hol}$. The potentials $c(t)$
might be not be normalized of course, even though they have nice regular properties.

\begin{enumerate}
\item The functions $c(t)$ are H\"{o}lder with constants bounded above by the maximum of the H\"{o}lder constants of $A$ and $B$, say $Q$; and H\"{o}lder exponents bounded below by some $\rho >0$ depending on $A,B$.
\item The projection $\bar{c}(t) = \Pi(c(t))$ is an analytic curve of the variables $A, B, t$, because the eigenfunctions
$h_{c(t)}$ and the eigenvalues $\lambda_{c(t)}$ are analytic functions of $c(t)$.
\item The curve of normalized potentials $\bar{c}(t)$ is a curve of H\"{o}lder functions with constants bounded by some $\bar{Q}$ and exponent bounded below by some $\delta$.
\item The radius of convergence of the series expansions in terms of $t$ of the functions $\bar{c}(t)$ around any $t_{0}$ is bounded below by some $D>0$ for every $t_{0} \in [0,1]$. This is because the radius of convergence of the series depends continuously on the parameter $t_{0} \in [0,1]$, so the compactness of $[0,1]$ implies that there is a positive lower bound for the radius of convergence.
\end{enumerate}

Let $\text{Hol}_{C,\alpha}(\Sigma)$ be the set of H\"{o}lder normalized potentials in $\Sigma = \{0, 1\}^{\mathbb{N}}$ with H\"{o}lder constant $C$, whose exponents are bounded below by $\alpha$, and let $\Gamma_{C,\alpha,\nu}$ be the family of curves $\upsilon: [0,1] \longrightarrow \text{Hol}_{C,\alpha}(\Sigma)$ depending analytically on $t \in [0,1]$, such that, the radius of convergence of the series expansion of the curves are bounded below by $\nu >0$. By Proposition \ref{Arzela-2}, we know that the family of functions $\Gamma_{C,\alpha,\nu}$ endowed with the sup norm is pre-compact.
\bigskip

\textbf{Proof of Theorem \ref{complete}}
\bigskip

Given $A, B \in \mathcal{N}$, we showed that the set of curves $\Gamma_{C,\alpha,\nu}$ is nonempty for certain values
of $C,\alpha,\nu$: the curve $\bar{c}(t) = \Pi(c(t))$ is in $\Gamma_{C,\delta,\nu}$. Therefore, either $\bar{c}(t)$ has minimal length in $\Gamma_{C,\alpha,\nu}$, and it is the geodesic we look for, or there exists a curve $c_{1}: [0,1] \longrightarrow \text{Hol}_{C, \alpha}(\Sigma)$ in $\Gamma_{C,\alpha,\nu}$ with strictly smaller length. By induction, either we find a geodesic $c_{n}$ in this process or we find a sequence $c_{k}$ of curves in $\Gamma_{C,\alpha,\nu}$ whose lengths converge to the infimum of the lengths of all curves joining $A, B$. By Proposition \ref{Arzela-2} there exists a
convergent subsequence whose limit is a continuous curve $\gamma_{\infty}: [-r,r] \longrightarrow \text{Hol}_{Q,\delta}$ , that is analytic in $t \in (-r,r)$, whose length attains the minimum of the lengths of curves joining $A,B$.

The curve $\gamma_{\infty}$ minimizes length in the family $\Gamma_{C,\alpha,\nu}$.
\bigskip

\textbf{Claim}: $\gamma_{\infty}$ is a true geodesic.
\bigskip

To show the Claim we apply the local existence results of the previous sections.

We know that there exists an open ball $B_{\rho}(A)$ around $A$ such that every nonzero tangent vector $X \in T_{A}\mathcal{N}$ determines uniquely a geodesic $\gamma_{X}: (-\rho, \rho) \longrightarrow \mathcal{N}$ such that $\gamma_{X}(0)= A$, $\gamma_{X}'(0)= X$. This local geodesic is an analytic curve of H\"{o}lder continuous functions whose H\"{o}lder constants are bounded above by a certain $\hat{C}$ and whose exponents are at least $\hat{\alpha}$. So if we replace $C$ by the maximum of $C, \hat{C}$, and $\alpha$ by the
minimum of $\alpha, \hat{\alpha}$, we get a precompact family of analytic curves $\Gamma_{C',\alpha',\nu}$ that contains the curve $\gamma_{\infty}$. Moreover, since $\gamma_{\infty}$ restricted to the ball $B_{\rho}(A)$ is a local minimizer in the family $\Gamma_{C, \alpha, \rho}$, Picard's Theorem \ref{Picard} and hence, the existence and uniqueness of local geodesics implies that $\gamma_{\infty}$ has to be
one of the solutions of the geodesic equation in the open ball $B_{\rho}(A)$.

This proves the Claim in an interval $[0,\rho)$ of the domain $[0, 1]$ of $\gamma_{\infty}$. If $\rho \geq 1$ then
we have shown that the curve is a geodesic as we wished. Otherwise, let us consider a local coordinate system at $P = \gamma_{\infty}(\rho)$ and let us look at the functions

$$ f_{Y}(t) = \langle \nabla_{\gamma_{\infty}'(t) } \gamma_{\infty}'(t), Y(\gamma_{\infty}(t)) \rangle $$
where $Y$ is an analytic vector field locally defined in the coordinate neighborhood of $P$. Since we know that
$\gamma_{\infty}(t)$ is analytic in $t$, as well as the Riemannian metric and the vector field $Y$, we have that the function $f_{Y}(t)$ is analytic in $t$. Since $\gamma_{\infty}$ is a geodesic in the interval $t \in (0, \rho)$, $f_{Y}(t) = 0$ for every $t \in (0,\rho)$, so we have by continuity that $f_{Y}(\rho) = 0$. The analyticity of $f_{Y}(t)$ then yields that there exists $\delta >0$ such that $f_{Y}(\rho + s) = 0$ for every $ \mid s \mid < \delta$. This shows that the curve $\gamma_{\infty}$ must be a geodesic in the whole interval $[0,1]$ as claimed.

\section{On the surface of Markov probabilities depending on two parameters} \label{aqui2}

We shall devote this section to the problem of the existence of geodesics on the surface of Markov probabilities. In the previous article \cite{LR1}, a detailed study of the Markov surface revealed remarkable geometric properties. Two of them are that the surface is totally geodesic in $\mathcal{N}$, and that its Gaussian curvature is zero everywhere. Let us recall the definition of the Markov surface and some of the main results about the intrinsic geometry of the surface in $\mathcal{N}$ obtained in \cite{LR1}.

Consider $M=\{0,1\}^\mathbb{N}$ and denote by $K$ the set of stationary Markov probabilities taking values in $\{0,1\}$.

Given a finite word $x =(x_1,x_2,...,x_k)\in \{0,1\}^k$, $k \in \mathbb{N}$, we denote by $[x]$ the associated cylinder set of size $k$ in
$\Omega=\{0,1\}^\mathbb{N}$.

Consider a shift invariant
Markov probability $\mu$ obtained from a row stochastic matrix $(P_{i,j})_{i,j=0,1}$ with positive entries and the initial left invariant vector of probability $\pi=(\pi_0,\pi_1)\in \mathbb{R}^2$. We denote by $A$ the H\"older potential associated to such probability $\mu$ (see Example 6 in \cite{LR}). There exists an explicit countable orthonormal basis, indexed by finite words $[x]$, for the set of H\"older functions on the kernel of the Ruelle operator $\op{L}_A$ (see \cite{LR1}).

Given $r \in (0,1)$ and $s\in (0,1)$ we denote
\begin{equation} \label{tororo2} P= \left(
\begin{array}{cc}
P_{0,0} & P_{0,1}\\
P_{1,0} & P_{1,1}
\end{array}\right)= \left(
\begin{array}{cc}
r & 1-r\\
1-s & s
\end{array}\right) .
\end{equation}

In this way $(r,s)\in (0,1) \times (0,1)$ parameterize all {\bf row} stochastic matrices we are interested. The following statement is proved in \cite{LR} and describes a special coordinate system for the surface $K$ of Markov probabilities.

\begin{figure}[!htb]
\centering
\hspace{-10pt} {\includegraphics[scale=0.6]{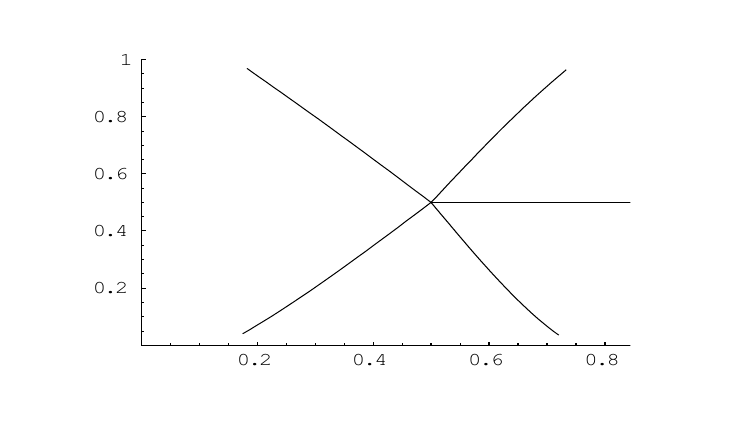}}
\caption{Numerical simulation - geodesics emanating from the point $(1/2,1/2)$ on parameter coordinates $(r,s)\in (0,1)\times (0,1)$ which describes the set of Markov probabilities.}
\label{esta3}
\end{figure}

\begin{theorem} \label{isothermal}
The Markov surface $K$ is totally geodesic in $\mathcal{N}$. Moreover, there exists a pair of unit vector fields $X_{1}$, $X_{2}$ tangent to $K$ which are orthogonal everywhere and satisfy the following properties: at a point of the stochastic matrix with coordinates $(r,s)$ we have
\begin{enumerate}
\item $\nabla_{X_{1}}X_{1} = \Gamma_{11}^{1} X_{1} $ where
$$ \Gamma_{11}^{1}= -\frac{(2r-1)(s-1)}{2(-2+r+s)} \frac{1}{(-\frac{(-1+r)r(-1+s)^{3}}{(-2+r+s)^{3}})^{\frac{1}{2}}}. $$
\item $\nabla_{X_{2}}X_{2} = \Gamma_{22}^{2} X_{2} $ where
$$ \Gamma_{22}^{2}= -\frac{(2s-1)(r-1)}{2(-2+r+s)} \frac{1}{(-\frac{(-1+s)s(-1+r)^{3}}{(-2+r+s)^{3}})^{\frac{1}{2}}}. $$
In particular, the vector fields $X_{1}$ and $X_{2}$ are geodesic vector fields, namely, their integral curves are
geodesics of $\mathcal{N}$.
\item $\nabla_{X_{1}}X_{2} = \nabla_{X_{2}}X_{1} = 0$, in particular, the vector fields $X_{1}$, $X_{2}$
commute and define a isothermal coordinate system for the Markov surface $K$.
\end{enumerate}
\end{theorem}

\begin{proof}
The Theorem is essentially proved in \cite{LR}. The only thing that deserves to be explained is the fact that the
vector fields $X_{1}$ and $X_{2}$ are geodesic. This is a well known result in the theory of geodesics: if a smooth vector field
$X$ satisfies $\nabla_{X} X= fX$, for a smooth scalar function $f$, then the integral orbits of $X$ are geodesics (see for instance \cite{Manfredo} 1979 Edition, Chapter 8, Lemma 3.1).
\end{proof}

\begin{figure}[!htb]
\centering
\hspace{-10pt} {\includegraphics[scale=0.6]{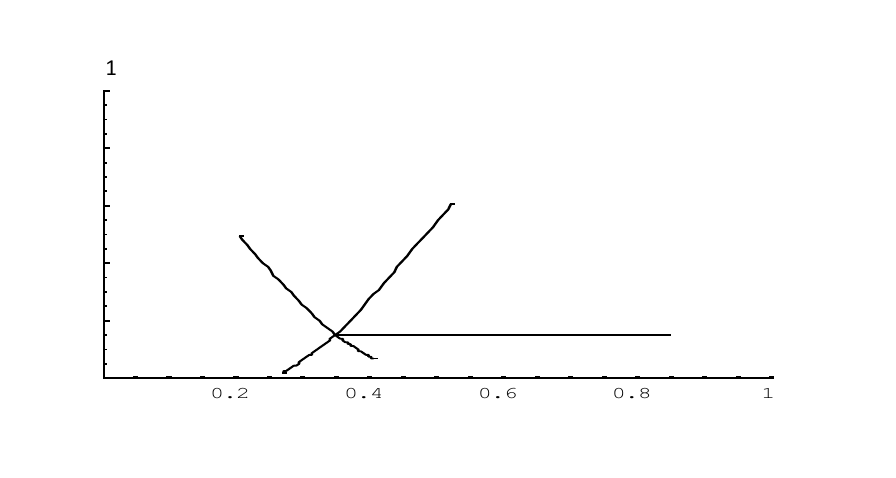}}
\caption{Numerical simulation - geodesics emanating from the point $(0.35,0.15)$ on parameter coordinates $(r,s)\in (0,1)\times (0,1)$ which describes the set of Markov probabilities.}
\label{esta4}
\end{figure}

The existence of an isothermal coordinate system is quite exceptional, and simplifies a great deal the system of
differential equations of geodesics in the surface. Moreover, it is easy to show that a surface with an isothermal coordinate system whose integral curves are geodesics is flat. Notice that the coefficients of the covariant derivatives in Theorem \ref{isothermal} are just the Christoffel coefficients of the coordinate system. In particular, item (1) implies that
$\Gamma_{11}^{2}=0$, item (2) that $\Gamma_{22}^{1} =0$, and item (3) that $\Gamma_{12}^{1}= \Gamma_{21}^{1}= \Gamma_{12}^{2}= \Gamma_{21}^{2} = 0$. The system of differential equations of geodesics in this coordinate system is then given by (see \cite{Manfredo} for instance )
$$ \frac{d^{2}u_{1}}{dt^{2}}(t) = \Gamma_{11}^{1}(u_{1}(t),u_{2}(t))(\frac{du_{1}}{dt}(t)) ^{2} $$
$$ \frac{d^{2}u_{2}}{dt^{2}}(t) = \Gamma_{22}^{2}(u_{1}(t),u_{2}(t))(\frac{du_{2}}{dt}(t)) ^{2} $$
where $\gamma(t) = (u_{1}(t), u_{2}(t))$ is the expression of a geodesic $\gamma(t)$ in the corresponding coordinates.
Note that the geodesics are not straight lines (one exception is the horizontal line through $r=1/2$). In figures 1 and 2, using Mathematica, we were able to show  parts of several geodesic paths with the initial position taken at the points, respectively, $(1/2,1/2)$ and $(0.35,0.15)$ .

\section{KL-divergence and dynamical information projections} \label{DIP}

Let us start with the second part of the article, focused on information projections in a dynamical context. First, we shall remind some preliminaries about KL-divergence and information theory.

\subsection{Introduction} \label{dop}

Through this section, the set
$M=\Omega =\{1,2,...,d\}^\mathbb{N}$, will be the compact symbolic space equipped with the usual metric $d$.

Recall that a Jacobian $J:\Omega \to(0,1)$ is a positive H\"older function such that
$\op{L}_{\log J}(1)=1,$ where $\op{L}_{\log J}$ is the the Ruelle operator for $\log J.$ The potential
$$x=(x_1,x_2,...,x_n,...) \to \log J(x_1,x_2,...,x_n,...)$$
is normalized.
Remember that for each Jacobian $J$ is associated a probability $\mu=\mu_J$, such that,
$\op{L}_{\log J}^*(\mu_J)= \mu_J$. The set of all possible $\mu_J$
constitutes the set $\mathcal{N}$.

A particular more simple example: if $\mu$ is a Markov shift invariant probability measure taking values in $\{1,2\}$, associated to a row stochastic matrix $P=(P_{i,j})_{i,j=1,2}$,
we get that the corresponding Jacobian $J$ satisfies $J(x) = P_{j,i}$, for each $x$ in the cylinder $\overline{i,j}\subset \{1,2\}^\mathbb{N}$ (see Example 6 in \cite{LR}). In this particular example the Jacobian $J$ depends only on the first two coordinates $x_1,x_2$ of $x=(x_1,x_2,x_3,...,x_n,...) \in \Omega=\{1,2\}^\mathbb{N}.$

The relation $ J \Longleftrightarrow \mu_J\in \mathcal{N}$ is bijective. It is natural to parameterize the H\"older equilibrium probabilities $\mu_J\in \mathcal{N}$ by the associated Jacobian $J$. We will not distinguish between naming $J$ and $\mu_J$.

The probabilities on $\mathcal{N}$ are ergodic for the action of the shift in the symbolic space $\Omega$ and hence they
are  singular with to respect to each other.  This property results, in some cases, in a certain difference when comparing our results and proofs with the non-dynamical  ones (as described in \cite{Ama}, \cite{Ay}, \cite{Nielsen1} and \cite{PolWu}).

Remember that we denote by $\text{Hol}$ the set of H\"older functions $f: \Omega \to \mathbb{R}$.

Here we are interested in the Kullback-Leibler divergence (KL-divergence for short) of shift invariant probabilities (see (17) in \cite{LM2})

Given two Jacobians $J_0$ and $J_1$ and the associated equilibrium probabilities $\mu_0$ and $\mu_1$ in $\mathcal{N}$,
its Kullback-Leibler divergence (or relative entropy) is given by
\begin{equation} \label{a1}
D_{KL}(\mu_0\,|\, \mu_1) =\int (\log J_0 - \log J_1)\,\, d \mu_0\geq 0.\,\,\,
\end{equation}

The above value is zero if and only if $\mu_0=\mu_1$ (which is the same as saying that $J_0=J_1$).
In some way, the relative entropy behaves like a kind of metric in the space of probabilities but the triangle inequality is not
always true (see \cite{Nielsen}). Moreover, the KL-divergence is not symmetric. $D_{KL}$ is convex in both variables.

\smallskip

Using the Riemannian structure of \cite{GKLM} and also \cite{Ji} it was shown in Section 5 in \cite{LR} that the Fisher information is equal to the asymptotic variance (see \cite{PP} for the definition). A result taken from \cite{LR}:

\begin{proposition} \label{pppr} Assume that $\xi$ is a tangent vector to $\mathcal{N}$ at $\mu_2$, then, for $\mu_1 \in \mathcal{N}$
\begin{equation} \label{china88} D_{K L} (\mu_1,\mu_ 2 + d \xi) = - \int \xi d \mu_1 + \frac{1}{2} \int \xi^2 d \mu_2 +
o (|d \xi|^2) ,
\end{equation}
where $ \int \xi^2 d \mu_2$ is the Fisher information.
\end{proposition}
\smallskip

Given Jacobians $J_0,\tilde{J}_1 \in \Theta_1$, $J_0 \neq \tilde{J}_1$,
consider the {\bf Jacobian} $\mathfrak{J}_\lambda$, $\lambda\in [0,1]$, such that,
\begin{equation} \label{a55} \mathfrak{J}_\lambda = J_0 + \lambda (\tilde{J}_1 - J_0).
\end{equation}

We denote by $\mu_{\mathfrak{J}_\lambda}$ the equilibrium probability associated to $\mathfrak{J}_\lambda$.  In this case $\mathfrak{J}_1= \tilde{J}_1.$

The probability $\mu_{\mathfrak{J}_\lambda}$ corresponds in \cite{Nielsen} to the concept of mixture distribution. In \cite{FLL} Bayesian Hypothesis Tests are considered for the family described by \eqref{a55}.

Note that in our dynamical setting the problem of considering a convex combination of probabilities is different from the problem of considering convex combinations of Jacobians like in \eqref{a55}.  As a non trivial convex combination of ergodic probabilities is not ergodic, it is more natural - under the ergodic point of view - to consider the family of probabilities $\mu_{\mathfrak{J}_\lambda}$ as described above in expression \eqref{a55}.

It follows from \cite{LR1} that the function $\frac{\tilde{J}_1-J_0}{J_0}$ corresponds to a tangent vector to the analytic manifold $\mathcal{N}$ at the point $\mu_{0}$.

The inequality
\begin{equation} \label{a61}\frac{d }{d \lambda} D_{KL}(\mu_{\mathfrak{J}_\lambda}, \mu_1)|_{\lambda=0} >0
\end{equation}
implies that the relative entropy of $\mu_{\mathfrak{J}_\lambda}$ with respect to $\mu_1$ is infinitesimally increasing on $\mu_0$
in the direction of $\tilde{J}_1-J_0.$ This can be consider a manifestation of the Second Law of Thermodynamics (see \cite{Cat} and Section 5 in \cite{LR}). Issues related to the derivative \eqref{a61} will be analyzed under the domain of what we will call
the Second Problem.
\smallskip

Given $\mu_1$ with Jacobian $J_1$, when $\mu_2=\mu_{\tilde{J}_1 }, \mu_0$ are such that
\begin{equation} \label{a612}\frac{d }{d \lambda} D_{KL}(\mu_{\mathfrak{J}_\lambda}, \mu_1)|_{\lambda=0} <0
\end{equation}
we say that this triple is under the fluctuation regime. The use of this terminology is in accordance with section 3.4 in \cite{Cat}. In the case \eqref{a61} is true we say that the triple
$\mu_0,\mu_1,\mu_2=\mu_{\tilde{J}_1 }$ is under the Second Law regime (a terminology borrowed from gas dynamics).

From a Bayesian point of view, the probability $\mu_1$ describes {\it the prior} probability and $\mu_{\mathfrak{J}_\lambda}$ plays the role of {\it the posterior} probability
in the inductive inference problem described by expression $D_{KL}(\mu_{\mathfrak{J}_\lambda}, \mu_1)$ (see Section 2.10 in
\cite{Cat}, \cite{LM2} and \cite{ELLM}).
The function $\log \mathfrak{J}_\lambda - \log J_1$ should be considered as the likelihood function (see \cite{FLL}).

\smallskip

Given $\mu_0=\mu_{J_0},\mu_1=\mu_{J_1},\mu_2=\mu_{\tilde{J}_1}$, the inequality
\begin{equation} \label{a2} D_{KL}(\mu_2, \mu_1) \geq D_{KL}(\mu_2, \mu_0) + D_{KL}(\mu_0, \mu_1)
\end{equation}
is called the {\bf Pythagorean inequality} (see Theorem 11.6.1 in \cite{Cover}).

Given $\mu_1$ with Jacobian $J_1$, expression \eqref{a2} is equivalent to
\begin{equation} \label{a225}\int (\log \tilde{J}_1 - \log J_1)\,\, d \mu_2 \geq \int (\log \tilde{J}_1 - \log J_0)\,\, d \mu_2 + \int (\log J_0 - \log J_1)\,\, d \mu_0.
\end{equation}

Interesting questions related to the Pythagorean inequality and information projections appear in Game Theory, Statistical Mechanics, Information Theory and Geometry (see \cite{Ga} \cite{EH}, \cite{Top}, \cite{KS}, \cite{Geo} and Section 12 in \cite{PolWu}).

A source of inspiration for our work is the following theorem presented in \cite{Cover}: given the probabilities $P=(P_1,P_2,...,P_d)$ and
$Q=(Q_1,Q_2,...,Q_d)$ on the set $\{1,2,...,d\}$, denote the KL divergence of $P$ and $Q$ by
$$ D(P\|Q) = \sum_{k=1}^d P_k \log P_k - \sum_{k=1}^d P_k \log Q_k.$$

Consider the probabilities $P^j=(P^j_1,P^j_2,...,P^j _d)$,
$j=0,1,2$, on $\{1,2,...,d\}$, and denote $P_\lambda= P^0 + \lambda (P^2 - P^0)$, $\lambda\in [0,1]$. Theorem 11.6.1 in \cite{Cover} claims that if
$$\frac{d}{d \lambda} D(P_\lambda, P_1)|_{\lambda =0} = \frac{d}{d \lambda}[\, \sum_{k=1}^d P_\lambda^k \log P_\lambda^k - \sum_{k=1}^d P_\lambda^k \log P^1_k ]|_{\lambda =0} >0 ,$$
then is true the {\it Pythagorean} inequality
\begin{equation} \label{cov225} D(P^2\|P^1) \geq D(P^2\|P^0) + D(P^0\|P^1) .
\end{equation}

In the case $\frac{d}{d \lambda} D(P_\lambda, P_1)|_{\lambda =0} \leq 0$ then the {\it triangle} inequality $ D(P^2\|P^1) \leq D(P^2\|P^0) + D(P^0\|P^1)$ is true.

Probabilities on $\{1,2,...,d\}$ have no dynamical content. Analogous results to the ones obtained in a non dynamical setting, when considered with respect to the dynamical setting of ergodic probabilities on $\{1,2,...,d\}^\mathbb{N} $, are not always true. The above probabilities $P^j$, $j=0,1,2$ are all absolutely continuous with respect to each other. The probabilities $\mu_j$,
$j=0,1,2$, described above are all singular with respect to each other. This makes a big difference when we want to demonstrate in our setting some analogous result which is known for the case of probabilities on $\{1,2,...,d\}$. Expression \eqref{cov225} for the probabilities $P^j=(P^j_1,P^j_2,...,P^j _d)$ on $\{1,2,...,d\}$,
$j=0,1,2$, corresponds in the dynamical setting to independent Bernoulli probabilities on $\Omega=\{1,2,...,d\}^\mathbb{N}.$ Example \ref{MaMa} in Section \ref{fi} shows that for a slightly more complex case, that it corresponds to consider Markov probabilities on $\{1,2,...,d\}^\mathbb{N}$, the analogous result to Theorem 11.6.1
in \cite{Cover} is not true.

Given $\mu_0=\mu_{J_0},\mu_1=\mu_{J_1},\mu_2=\mu_{\tilde{J}_1}$, the alternative inequality to \eqref{a2} is
\begin{equation} \label{a3} D_{K L} (\mu_2,\mu_1) \leq D_{K L} (\mu_2,\mu_0) + D_{K L} (\mu_0,\mu_1),
\end{equation}
which is known as the {\bf triangle inequality}.

A natural question to ask is when these inequalities appear when considering some extremality property regarding a fixed convex set of probabilities (like expression \eqref{a4} in the First problem to be defined next).

\medskip

Consider $\Theta_1\subset \text{Hol}$ a convex compact set of H\"older Jacobians $J:\Omega \to (0,1)$. Note that given the Jacobians $J_0,J_1$, the convex combination
\begin{equation} \label{a334}\lambda J_0 + (1- \lambda) J_1
\end{equation}
is also a Jacobian. From the bijective relation $ J \Longleftrightarrow \mu_J\in \mathcal{N}$ one can see
$\Theta_1$ as a subset of $\mathcal{N}.$

Note that given the Jacobians $J_0,J_1$, the convex combination $ \lambda \log (J_0) + (1- \lambda) \log (J_1)$ {\bf is not} of the form
$\log J$, for a Jacobian $J$.

\smallskip

{\bf First problem}: given the fixed H\"older Jacobian $J_1\notin \Theta_1$ (associated to a probability $\mu_1=\mu_{J_1}$) assume that the Jacobian $J_0\in \Theta_1$ minimize Kullback-Leibler divergence, that is, $\mu_{J_0}=\mu_0$ satisfies
\begin{equation} \label{a4} D_{KL}(\mu_1, \mu_0)= \min_{\tilde{J}\in \Theta_1} D_{KL}( \mu_1, \mu_{\tilde{J}}).
\end{equation}

A $\tilde{J}=J_0$ minimizing \eqref{a4} will be called a solution of the {\bf minimizing first problem} (Chapter 3 in \cite{Iha} also investigate similar problems).
We analyze here the information projection problem for equilibrium probabilities.

One can also analyze the related maximizing problem
\begin{equation} \label{c44} \max_{\tilde{J}\in \Theta_1} D_{KL}( \mu_1, \mu_{\tilde{J}}).
\end{equation}
with similar methods.
\smallskip

A $\tilde{J}=J_0$ maximizing \eqref{c44} will be called a solution of the {\bf maximizing first problem}.

Given $J_0,\tilde{J}_1 \in \Theta_1$, $J_0 \neq \tilde{J}_1$, we
denote by $\mathfrak{J}_\lambda\in \Theta_1$, $\lambda\in [0,1]$, the function
\begin{equation} \label{a51} \mathfrak{J}_\lambda = \lambda \tilde{J}_1 + (1- \lambda) J_0,
\end{equation}
where we denote by $\mu_{\mathfrak{J}_\lambda}$ the equilibrium probability associated to the Jacobian $\mathfrak{J}_\lambda.$

When $J_0\in \Theta_1$  minimize the Kullback-Leibler divergence in problem \eqref{a4}, and $\mathfrak{J}_\lambda$ satisfies \eqref{a51}, we get for any $\tilde{J}_1\in \Theta_1$
\begin{equation} \label{a62}\frac{d }{d \lambda} D_{KL}(\mu_1, \mu_{\mathfrak{J}_\lambda})|_{\lambda=0} =\frac{d}{d \lambda}[\,\int \log J_1 d \mu_1 - \int \log \mathfrak{J}_\lambda d
\mu_1\,]\,|_{\lambda=0}\geq 0.
\end{equation}

When $J_0\in \Theta_1$ maximize \eqref{c44} we get  for any $\tilde{J}_1\in \Theta_1$
\begin{equation} \label{a611}\frac{d }{d \lambda} D_{KL}(\mu_1, \mu_{\mathfrak{J}_\lambda})|_{\lambda=0} =\frac{d}{d \lambda}[\,\int \log J_1 d \mu_1 - \int \log \mathfrak{J}_\lambda d
\mu_1\,]\,|_{\lambda=0}\leq 0.
\end{equation}

A more useful information is the exact estimate of the value in \eqref{a62} and \eqref{a611} given by \eqref{a49} (to be obtained in section \ref{tre}).

\begin{proposition} \label{er2377}
\begin{equation} \label{a49} \frac{d}{d \lambda}|_{\lambda=0}D_{KL}(\mu_1, \mu_{\mathfrak{J}_\lambda})=\int (1- \frac{\tilde{J}_1}{ J_0}) d \mu_1.
\end{equation}
\end{proposition}

The above value can be positive or negative in different cases. Note that taking $\tilde{J}_1$ and $J_0$ more and more close to each other will imply that the derivative at $\lambda=0$ is closer and closer to zero.

\begin{remark} \label{kak} Assume $\Theta_1$ is a convex simplex generated by $\mathcal{J}_r, r=1,2,...,w$, in the maximization problem \eqref{c44}.
We can ask how to characterize an optimal $J_0$.
As $D_{K L}$ is convex in both variables, the Jacobian $J_0$ (one of the possibles $\mathcal{J}_r, r=1,2,...,w$ ) should be in the boundary of $\Theta_1$

In this way it is natural to consider
\begin{equation} \label{a512} \mathfrak{J}_\lambda^r = \lambda J_0 + (1- \lambda) \mathcal{J}_r\,\, r=1,2,...,w,
\end{equation}
and the associated $\mu_{\mathfrak{J}_\lambda}^r$.

From \eqref{a49} and \eqref{c44} we get that in the second maximization problem
the Jacobian $J_0$ should satisfy the equations
\begin{equation} \label{a149} \int (1- \frac{ J_0}{\mathcal{J}_r}) d \mu_1\geq 0,\,\,\, r=1,2,...,w.
\end{equation}
In this way we have just a finite number of
inequalities to check.
\end{remark}

\medskip

Consider a fixed probability $\mu_1$ with  Jacobian $\mu_1$.
It is also natural to analyze the first type of problem on $\Theta_2$. 
In this way one should a consider the probability $\mu^\lambda$, $\lambda \in [0,1]$, that is the  equilibrium probability  for the potential
\begin{equation} \label{a1010} \lambda \log (\tilde{J}_1) + (1- \lambda) \log (J_0),
\end{equation}
$\lambda\in [0,1]$. We denote {\bf  by $\mathfrak{J}^\lambda$ the Jacobian of the equilibrium probability $\mu^\lambda$} for the potential $\lambda \log (\tilde{J}_1) + (1- \lambda) \log (J_0)$ ({\bf $\mathfrak{J}_\lambda$ is different from $\mathfrak{J}^\lambda$}).
The probability $\mu^1$ has Jacobian $\tilde{J}_1= \mathfrak{J}^1$ and the probability $\mu^0$ has Jacobian $J_0=\mathfrak{J}^0$.

In this case the minimization of
$D_{KL}(\mu_1, \mu_{\tilde{J}})$ will be considered over the set of $\tilde{J}\in \Theta_2$.

Note that the Riemannian manifold of H\"older equilibrium probabilities $\mathcal{N}$ is not flat (see \cite{GKLM} and \cite{LR1}).
Adapting the terminology of \cite{Nielsen} for our dynamical setting, it is natural to call $\log \mathfrak{J}^\lambda$
the linear interpolation of $\mu^0$ an $\mu^1$ at $\lambda$ on the logarithm scale.

The pressure problem for potentials of the form \eqref{a1010} is considered in expression (3.27) in \cite{FLL} (where a different notation is used). More precisely, in the notation we consider here set
\begin{equation} \label{equio} P_1( \lambda)= P(\lambda (\log \tilde{J}_1 - \log J_0) + \log J_0),
\end{equation}
where $\lambda\in[0,1]$ and $P(A)$ denotes the pressure of the potential $A$ (see \cite{PP}). In this case $P_1(0)=0=P_1(1)$, and from expression (3.30) in \cite{FLL}, we get
\begin{equation} \label{equio1}P_1^{\,\prime} (0)= \int (\log \tilde{J}_1 - \log J_0)\,d \mu^0.\end{equation}

The function $\lambda \to P_1 (\lambda)$ described by expression \eqref{equio} corresponds here to the integral-based Bregman generator
(see (156) in \cite{Nielsen1})

Taking $E=0$ in expression (3.36) in \cite{FLL} we get (in the present notation) the deviation function (a Legendre transform)
\begin{equation} \label{equio2} P_1^* (\eta) = \sup_{\lambda \in [0,1]} \, \{\lambda\, \eta \,- P_1(\lambda ) \}.
\end{equation}

Following the reasoning of \cite{Nielsen} it is natural to call
\begin{equation} \label{equio3} B_{P_1} = P_1(1) - P_1(0) - (1 - 0) P_1^{\,\prime} (0)=
\int (\log J_0 - \log \tilde{J}_1)\,d \mu^0 > 0
\end{equation}
the {\it Bregman divergence} for $\mu^0$ an $\mu^1$ (which is this case is equal to $D_{KL}(\mu^0\,|\, \mu^1)$).

\medskip

We will show in Section \ref{to} that
\begin{proposition} Given $\mu_1=\mu_{J_1}$
\begin{equation} \label{chili0}\frac{d}{d \lambda}|_{\lambda=0} D_{K L} (\mu_1,\mu^\lambda) = - \int (\log \tilde{J}_1 - \log J_0) d \mu_1 + \int (\log \tilde{J}_1 - \log J_0) \,d \mu^0.
\end{equation}

The inequality $0\leq \frac{d}{d \lambda}|_{\lambda=0} D_{K L} (\mu_1,\mu^\lambda) $ is equivalent to the Pythagorean inequality:

$$ D_{K L} (\mu_1,\mu^0) + D_{K L} (\mu^0,\mu^2) \leq D_{K L} (\mu_1,\mu^2),$$
where $\mu_2=\mu^1 $
\end{proposition}
\medskip

{\bf Second problem}: given the fixed H\"older Jacobian $J_1\notin \Theta_1$ (associated to a probability $\mu_1=\mu_{J_1}$) assume that $\tilde{J}=J_0\in \Theta_1$ minimize Kullback-Leibler divergence, that is, $\mu_{J_0}=\mu_0$ satisfies
\begin{equation} \label{a7} D_{KL}(\mu_0,\mu_1)= \min_{\tilde{J}\in \Theta_1} D_{KL}( \mu_{\tilde{J}}, \mu_1).
\end{equation}


A $J_0$ minimizing \eqref{a7} will be called a solution of the minimizing second problem.

A $J_0$ maximizing
\begin{equation} \label{a77} D_{KL}(\mu_0,\mu_1)= \max_{\tilde{J}\in \Theta_1} D_{KL}( \mu_{\tilde{J}}, \mu_1).
\end{equation}
will be called a solution of the maximizing second problem.

Given $\mu_1=\mu_{J_1}$ with Jacobian $J_1$,  related to the minimizing second problem we have the following inequality: given a Jacobian $\tilde{J}_1 \in \Theta_1$ and
$\mathfrak{J}_\lambda$ as above
\begin{equation} \label{a8}\frac{d}{d \lambda}D_{KL}(\mu_{\mathfrak{J}_\lambda}, \mu_1)|_{\lambda=0} =\frac{d}{d \lambda}[\,\int \log \mathfrak{J}_\lambda d \mu_{\mathfrak{J}_\lambda} - \int \log J_1
\mu_{\mathfrak{J}_\lambda}\,]\,|_{\lambda=0}\geq 0.
\end{equation}

In this case, we are in the Second Law of Thermodynamics regime.

The second problem is harder than the first one. In Section \ref{fi} we estimate the value in \eqref{a8}:
\begin{proposition} \label{era8812}Given $\mu_1=\mu_{J_1}$ with Jacobian $J_1$
\begin{equation} \label{a83469}\frac{d}{d \lambda}D_{KL}(\mu_{\mathfrak{J}_\lambda}, \mu_1)|_{\lambda=0}= \int (\log J_0- \log J_1) \,( \frac{\tilde{J}_1-J_0}{J_0} )d \mu_{0}
\end{equation}
\end{proposition}

Remember that the function $\xi=\frac{\tilde{J}_1-J_0}{J_0}$ corresponds to a tangent vector to the analytic manifold $\mathcal{N}$ at the point $\mu_{0}$.

A natural question is the following: for $J_0,J_1$ fixed, is there a direction $\frac{\tilde{J}_1-J_0}{J_0}$ where the derivative $\frac{d}{d \lambda}D_{KL}(\mu_{\mathfrak{J}_\lambda}, \mu_1)|_{\lambda=0}$ is maximal? This requires explicit expressions for this derivative.


\medskip

Via a counterexample in section \ref{fi} we will show that not always the inequality $\frac{d}{d \lambda}D_{KL}(\mu_{\mathfrak{J}_\lambda}, \mu_1)|_{\lambda=0}\geq 0$ implies the Pythagorean inequality

$$ D_{K L} (\mu_2,\mu_0) + D_{K L} (\mu_0,\mu_1) \leq D_{K L} (\mu_2,\mu_1),
$$
when $\mu_2=\mu_{\tilde{J}_1}$
\medskip

\medskip
Once more it makes sense to analyze the second type of problem on $\tilde{J}\in \Theta_2$, considering the family $\mu^\lambda, 0 \leq \lambda \leq 1$, which is the equilibrium probability for the potential $ \lambda \log (\tilde{J}_1) + (1- \lambda) \log (J_0).$
\medskip

In Section \ref{qua} we show that
\begin{proposition} \label{eraa8812b} Given $\mu_1$ with Jacobian $J_1$
$$\frac{d}{d \lambda}D_{KL}(\mu^\lambda, \mu_1)|_{\lambda=0}= \int (\log J_0- \log J_1) \,(\,\log (\tilde{J}_1) - \log( J_0)\,) \,d \mu_{0}.$$
\end{proposition}
\smallskip

One can also analyze the related problem
\begin{equation} \label{c4} \max_{ \tilde{J}\in \Theta_1} D_{KL}( \mu_{\tilde{J}}, \mu_1).
\end{equation}
with similar methods to the ones used below.

\medskip

When investigating properties related to
\begin{equation} \label{a9}\lambda \to \mathfrak{J}_\lambda = \lambda \tilde{J}_1 + (1- \lambda) J_0,
\end{equation}

\noindent
we say we are considering a
{\bf $J$-case}.

\medskip

On the other hand when investigating properties related to
\begin{equation} \label{a10}\lambda \to \lambda \log (\tilde{J}_1) + (1- \lambda) \log (J_0),
\end{equation}

\noindent
we say we are considering a
{\bf $\log J$-case}.

\medskip

Given the H\"older potential $A$, in this section we denote by $\alpha_A$ and $\varphi_A$, respectively, the main eigenvalue and the main eigenfunction of the Ruelle operator $\op{L}_A.$ Using the notation of \cite{GKLM},  remember that we denote by $\Pi$ the normalization map
\begin{equation} \label{k1}
\Pi(A) = A + \log \varphi_A - (\log \varphi_A \circ \sigma) - \log \alpha_A.
\end{equation}

The equilibrium probability $\mu_A$ for $A$ has Jacobian $e^{\Pi(A)}$. It is know that $\mu_A= \mu_{\Pi(A)}$.
We are interested in the perturbed potential $A+ \xi$ for very small $\xi$. The entropy of $\mu_A$ is equal
to $- \int \Pi(A) d \mu_A$ (see Corollary 5.3 in \cite{GKLM}).

Given the H\"older potential $A$ the function $D \Pi(A) (\xi)$ is the projection of $\xi$ in the kernel of $\op{L}_A.$
Moreover,
$$D \Pi(A) (\xi)= \xi - \int \xi d \mu_A + u - (u \circ \sigma),$$
for some continuous function $u: \Omega \to \mathbb{R}$ (see (5) section 4.2 in \cite{GKLM})
\smallskip

An important property that we will use here is the following: denote $\mathfrak{h}_t$ the entropy of $\mu_{A + t \xi}$. From section 7.3.1 in \cite{GKLM} we get that
\begin{equation} \label{k2}
\frac{d}{d t} \mathfrak{h}_t|_{t=0} = - \int D \Pi (A) (\xi) \ d\mu_A - \int \Pi(A) \, \xi \,d \mu_A = -\int \Pi(A) \xi d \mu_A.
\end{equation}

\smallskip

Below we summarize the most important relationships that will be used next.

\smallskip

{\bf I}. In order to show the type-1 Pythagorean inequality, we have to show that given $\mu_1=\mu_{J_1},\mu_2=\mu_{J_2},\mu_0=\mu_{J_0}$

$$ D_{K L} (\mu_2,\mu_0) + D_{K L} (\mu_0,\mu_1) = \int (\log J_2 - \log J_0) d \mu_2 + \int (\log J_0 - \log J_1) d \mu_0 \leq $$
\begin{equation} \label{pi1} \int (\log J_2 - \log J_1) d \mu_2= D_{K L} (\mu_2,\mu_1).
\end{equation}

When $\mu_2=\mu_{\tilde{J}_1}$, this is equivalent to
$$0 \leq \int (\log J_0 - \log J_1) d \mu_2 - \int (\log J_0 - \log J_1) d \mu_0= $$
\begin{equation} \label{pia1}  \int (\log J_0 - \log J_1) d \mu_{\tilde{J}_1} - \int (\log J_0 - \log J_1) d \mu_0 .
\end{equation}

{\bf II}. In order to show the type-2 Pythagorean inequality, we have to show that

$$ D_{K L} (\mu_1,\mu_0) + D_{K L} (\mu_0,\mu_2) = \int (\log J_1 - \log J_0) d \mu_1 - \int (\log J_0 - \log J_2) d \mu_0 \leq $$
\begin{equation} \label{pi2} \int (\log J_1 - \log J_2) d \mu_1= D_{K L} (\mu_1,\mu_2),
\end{equation}
where $\mu_2$ has Jacobian $J_2$.

When $\mu_2=\mu_{\tilde{J}_1}$, this is equivalent to
\begin{equation} \label{pia2} 0 \leq \int (\log J_0 - \log \tilde{J}_1) d \mu_1 - \int (\log J_0 - \log \tilde{J}_1) d \mu_0.
\end{equation}

{\bf III}. The type-1 triangle inequality
$$ D_{K L} (\mu_2,\mu_0) + D_{K L} (\mu_0,\mu_1) \geq D_{K L} (\mu_2,\mu_1) $$
is equivalent to

$$ \int (\log J_0 - \log J_1) d \mu_0 \geq \int (\log J_0 - \log J_1) d \mu_2. $$

{\bf IV}. The type-2 Pythagorean inequality is
$$ D_{K L} (\mu_1,\mu_0) + D_{K L} (\mu_0,\mu_{\tilde{J}_1 }) = \int (\log J_1 - \log J_0) d \mu_1 - \int (\log J_0 - \log \tilde{J}_1) d \mu_0 \geq $$
\begin{equation} \label{pi4} \int (\log J_1 - \log \tilde{J}_1) d \mu_1= D_{K L} (\mu_1,\mu_{\tilde{J}_1 }).
\end{equation}

The above is equivalent to
\begin{equation} \label{pia4} 0 \leq \int (\log J_0 - \log \tilde{J}_1) d \mu_1 - \int (\log J_0 - \log \tilde{J}_1) d \mu_0.
\end{equation}

\subsection{The $\log J$ case}

\subsubsection{First problem} \label{to}

$\Theta_2$ denotes the convex set of H\"older potentials $\tilde{J}$ that were defined in Subsection \ref{dop}.
Consider an H\"older Jacobian $J_0 $ associated to an H\"older potential $A_0 \in \Theta_2,$ and $\mu_0$ the associated equilibrium probability.
Denote $A_t =\log J_0 +\, t \,\xi$, where $\xi$ is a tangent vector at $\mu_0$ and $t \in \mathbb{R}$. We assume that $\xi$ is such that $A_t$ belongs to $\Theta_2$, for all $t\in[0,1]$. When $\xi= \log \tilde{J}_1 - \log J_0$, the associated H\"older Jacobian is denoted by $\mathfrak{J}^t$ and $\mu^t$ is the associated equilibrium state for $A_t$ (or, for $\log \mathfrak{J}^t$).

\smallskip

{\bf A minimization problem:} suppose $\mu_1$ with Jacobian $J_1$ is fixed and $J_1 \notin \Theta_2$. Suppose that $J_0$ is the Jacobian of a certain special potential $A_0$ in $\Theta_2$.

Here the probability $\mu^t$, $\lambda \in [0,1]$,  is the  equilibrium probability  for the potential
\begin{equation} \label{a10107} t \log (\tilde{J}_1) + (1- t) \log (J_0),
\end{equation}
$t\in [0,1]$. We denote {\bf  by $\mathfrak{J}^t$ the Jacobian of the equilibrium probability $\mu^t$} for the potential $t \log (\tilde{J}_1) + (1- t) \log (J_0)$.

We will assume that $J_0$ satisfies an extremality property described in the following way: given any Jacobian $\tilde{J}_1$, associated to a potential $A_2$ in $\Theta_2$, denote $g:[0,1] \to \mathbb{R}$ by
\begin{equation} \label{a11}t \to g(t)= D_{K L} (\mu_1,\mu^t),
\end{equation}
when $\xi= \log \tilde{J}_1 - \log J_0$. Under our hypothesis $A_t$ belongs to $\Theta_2$, for $t\in[0,1]$.

Note that $g(0) = D_{K L} (\mu_1,\mu_0)$ and, as $\mu^1=\mu_{\tilde{J}_1}$, $g(1) = D_{K L} (\mu_1,\mu^1)$

The extremality property for $J_0$ is that $g(t)$ has a  minimum at $0$. This implies that

\begin{equation} \label{a12} \frac{d}{d t}|_{t=0} D_{K L} (\mu_1,\mu^t) =\frac{d}{d t}|_{t=0} ( \int \log J_1 d \mu_1 - \int \log \mathfrak{J}^t d \mu_1) \geq 0.
\end{equation}

The above means that in some sense we are taking as $\mu_0$ the {\it $D_{KL}$-closest} probability to $\mu_1$ in $\Theta_2$.

There exist $\varphi_t$ and $\lambda_t \in \mathbb{R}$, such that, the Jacobian $\mathfrak{J}^t$ satisfies
\begin{equation} \label{rewo12}\log \mathfrak{J}^t= \log J_0 + t \, \xi + \log \varphi_t - \log \varphi_t (\sigma) - \log \alpha_t,
\end{equation}
where $\log \alpha_t= P(\log J_0 + t \, \xi ).$

It is known (see \cite{PP}) that for a continuous function $\xi :\Omega \to \mathbb{R}$ (not necessarily satisfying $\int \xi d \mu_0=0$)
\begin{equation} \label{rewo} \frac{d}{d t}\log \alpha_t|_{t=0} = \frac{d}{d t} P(\log J_0 + t \,\xi )|_{t=0} =\int \xi d \mu_0.
\end{equation}

From the invariance of $\mu_1$
$$ 0\leq \frac{d}{d t}|_{t=0} D_{K L} (\mu_1,\mu^t) = $$
$$ \frac{d}{d t}|_{t=0} \,[ \int \log J_1 d \mu_1 - \int (\log J_0 + \theta \, \xi + \log \varphi_t- \log \varphi_t (\sigma) - \log \alpha_t ) d \mu_1 ]=$$
$$ \frac{d}{d t}|_{t=0} \,[ \int \log J_1 d \mu_1 - \int (\log J_0 + t \, \xi - \log \alpha_t ) d \mu_1 ]=$$
\begin{equation} \label{china11} - \int \xi d \mu_1 + \int \xi d \mu_0 .
\end{equation}

Therefore, taking $\xi= \log \tilde{J}_1 - \log J_0$
$$ 0\leq \frac{d}{d t}|_{t=0} D_{K L} (\mu_1,\mu^t) = $$
\begin{equation} \label{china0} - \int (\log \tilde{J}_1 - \log J_0) d \mu_1 + \int (\log \tilde{J}_1 - \log J_0) d \mu_0.
\end{equation}

This is equivalent to the type-2 Pythagorean inequality:

$$ D_{K L} (\mu_1,\mu_0) + D_{K L} (\mu_0,\mu^1) = \int (\log J_1 - \log J_0) d \mu_1 + \int (\log J_0 - \log \tilde{J}_1) d \mu_0 \leq $$
$$ \int (\log J_1 - \log \tilde{J}_1) d \mu_1= D_{K L} (\mu_1,\mu^1). $$

\medskip

{\bf A maximization problem for $\tilde{J}\in \Theta_2$:} a similar problem will produce the triangle inequality. Suppose $\mu_1$ with Jacobian $J_1$ is fixed and $\log J_1 \notin \Theta_2$.
Now, we assume that $\tilde{J}=J_0$ satisfies a different extremality property described in the following way: in the same way as before take a Jacobian $\tilde{J}_1$ associated to a potential $A_2$ in $\Theta_2$, and denote $g:[0,1] \to \mathbb{R}$ by
$$ g(t)= D_{K L} (\mu_1,\mu^t), $$
when $\xi= \log \tilde{J}_1 - \log J_0$.

The new extremality property for $J_0$ is that $g(t)$ has  maximum at $0$. This implies that

\begin{equation} \label{lut} \frac{d}{d t}|_{t=0} D_{K L} (\mu_1,\mu^t) =\frac{d}{d t}|_{t=0} ( \int \log J_1 d \mu_1 - \int \log \mathfrak{J}^t d \mu_1) \leq 0.
\end{equation}

The above means that in some sense we are taking as $\mu_0$ the {\it more $D_{KL}$-distant} probability to $\mu_1$ in $\Theta_2$.

Using \eqref{china11} one can show the triangle inequality
$$ D_{K L} (\mu_1,\mu_0) + D_{K L} (\mu_0,\mu_2) \geq D_{K L} (\mu_1,\mu_2),$$
where $\mu_2=\mu^1$.

\medskip

\subsubsection{Second problem} \label{qua}

In this section we consider the family
\begin{equation} \label{a101} \lambda \log (\tilde{J}_1) + (1- \lambda) \log (J_0)= \log J_0 + \lambda (\,\log (\tilde{J}_1) - \log( J_0)\,),
\end{equation}
where $\lambda \in [0,1].$

Note that
$$ \frac{d}{d \lambda} [\,\log (J_0) + \lambda (\,\log (\tilde{J}_1) - \log( J_0)\,)\,] \,=\,\log (\tilde{J}_1) - \log( J_0)\, . $$

Remember that $\mathfrak{J}^\lambda$ is the Jacobian of the equilibrium probability for $ \lambda \log ( \tilde{J}_1) - (1 - \lambda) \log (J_0)$.

We denote by $\mu^\lambda=\mu_{ \mathfrak{J}^\lambda }$ the equilibrium probability for $\log \mathfrak{J}^\lambda$. The Shannon-Kolmogorov entropy of
$\mu^\lambda$ is $-\int \log \mathfrak{J}^\lambda d \mu^\lambda$.

We want to estimate
$$\frac{d}{d \lambda}|_{\lambda=0}D_{KL}(\mu^\lambda, \mu_1).$$

From Theorem 5.1 in \cite{GKLM} we get
\begin{proposition} \label{erte88b} Denote by $\mathfrak{J}^\lambda$ the Jacobian of the equilibrium probability for the potential $ \lambda \log ( \tilde{J}_1) - (1 - \lambda) \log (J_0)$. Then,
\begin{equation} \label{b1b}
\frac{d}{d \lambda}\int \log J_1 \, d \mu_{ \mathfrak{J}^\lambda}|_{\lambda =0}=\int \, \log J_1\, (\,\log (\tilde{J}_1) - \log( J_0)\,) \,\, d \mu_{0}.
\end{equation}
\end{proposition}

From \eqref{k2} we get
\begin{proposition} \label{erte881b} Denote by $\mathfrak{J}^\lambda$ the Jacobian of the equilibrium probability for $ \lambda \log ( \tilde{J}_1) - (1 - \lambda) \log (J_0)$. Then, the derivative of minus the entropy of $\mu_{ \mathfrak{J}^\lambda}$ is
\begin{equation} \label{bo1b} \frac{d}{d \lambda}\int \log \mathfrak{J}_\lambda \, d \mu_{ \mathfrak{J}^\lambda}|_{\lambda =0}= \int \log J_0 (\,\log (\tilde{J}_1) - \log( J_0)\,) \,d \mu_{0}.
\end{equation}
\end{proposition}

From \eqref{bo1b} it follows
\begin{proposition} \label{erte8812b}
\begin{equation} \label{a834}\frac{d}{d \lambda}D_{KL}(\mu^\lambda, \mu_1)|_{\lambda=0}= \int (\log J_0- \log J_1) \,(\,\log (\tilde{J}_1) - \log( J_0)\,) \,d \mu_{0}.
\end{equation}
\end{proposition}

In the case $\frac{d}{d \lambda}D_{KL}(\mu^\lambda, \mu_1)|_{\lambda=0}>0$ we get that
\begin{equation} \label{aa834} \int (\log J_0- \log J_1) \,(\,\log (\tilde{J}_1) - \log( J_0)\,) \,d \mu_{0}\geq 0.
\end{equation}



\subsection{The $J$ case}

\subsubsection{Second problem} \label{fi}

In this section we consider the family of Jacobians
\begin{equation} \label{a99} \mathfrak{J}_\lambda = \lambda \tilde{J}_1 + (1- \lambda) J_0,
\end{equation}
$\lambda\in [0,1].$ The probability $\mu_{\mathfrak{J}_\lambda}$ is the one with Jacobian $\mathfrak{J}_\lambda$. We use the notation $\mathfrak{J}_\lambda$ to distinguish from $\mathfrak{J}^\lambda$, which was used in last subsection.

\smallskip

We will show that
$$ \frac{d}{d \lambda}D_{KL}(\mu_{\mathfrak{J}_\lambda}, \mu_1)|_{\lambda=0}= \int (\log J_0- \log J_1) \,( \frac{\tilde{J}_1-J_0}{J_0} )d \mu_{0}.$$
\smallskip

Note that for any $x\in \Omega$
\begin{equation} \label{triu}
\op{L}_{\log J_0} ( 1 - \frac{\tilde{J}_1}{J_0} )(x) = 1- \sum_{a} J_0(a x) \frac{\tilde{J}_1(a x)}{J_0(a x)}= 1 - \sum_{a} \tilde{J}_1(a x)=1-1=0
\end{equation}

Then, the function $1 - \frac{\tilde{J}_1}{J_0}$ is in the kernel of the operator $\op{L}_{\log J_0},$
and
\begin{equation} \label{tritu7}
\int (1 - \frac{\tilde{J}_1}{J_0} ) d \mu_{0}=0.
\end{equation}

Therefore, the function $1 - \frac{\tilde{J}_1}{J_0}$ is a tangent vector to the manifold $\mathcal{N}$ at the point $\mu_{0}$ (see \cite{GKLM}).

Note that
$$\frac{d}{d \lambda} \log \mathfrak{J}_\lambda = \frac{\tilde{J}_1 - J_0 } {\mathfrak{J}_\lambda}.$$

The type-1 Pythagorean inequality
$$ D_{K L} (\mu_2,\mu_0) - ( D_{K L} (\mu_2,\mu_0) + D_{K L} (\mu_0,\mu_1))\geq 0$$
is equivalent to
\begin{equation} \label{Pyt7} \int (\log J_1 - \log J_0) d \mu_0 + \int (\log J_0 - \log J_1) d \mu_2 \geq 0.
\end{equation}

The type-2 Pythagorean inequality is equivalent to
\begin{equation} \label{pia4} 0 \leq \int (\log J_0 - \log \tilde{J}_1) d \mu_1 - \int (\log J_0 - \log \tilde{J}_1) d \mu_0.
\end{equation}

From Theorem 5.1 in \cite{GKLM} we get
\begin{proposition} \label{erte88} Denote by $\mathfrak{J}_\lambda$ the Jacobian $ \mathfrak{J}_\lambda =\lambda \tilde{J}_1 - (1 - \lambda) J_0$. Then,
\begin{equation} \label{b1}
\frac{d}{d \lambda}\int \log J_1 \, d \mu_{ \mathfrak{J}_\lambda}|_{\lambda =0}=\int \, \log J_1\, \frac{( \tilde{J}_1 - J_0)}{J_0} \,\, d \mu_{0}
\end{equation}
and
\begin{equation} \label{bb1}
\frac{d}{d \lambda}\int \log J_1 \, d \mu_{ \mathfrak{J}_\lambda}|_{\lambda =1}=\int \, \log J_1\, \frac{( \tilde{J}_1 - J_0)}{J_0} \,\, d \mu_{\tilde{J}_1}
\end{equation}
\end{proposition}

From \eqref{k2} we get
\begin{proposition} \label{erte881} Denote by $\mathfrak{J}_\lambda$ the Jacobian $\mathfrak{J}_\lambda = \lambda \tilde{J}_1 - (1 - \lambda) J_0$. Then, the derivative of minus the entropy of $\mu_{\log \mathfrak{J}_\lambda}$
\begin{equation} \label{bo1} \frac{d}{d \lambda}\int \log \mathfrak{J}_\lambda \, d \mu_{\mathfrak{J}_\lambda}|_{\lambda =0}= \int \log J_0 ( \frac{\tilde{J}_1-J_0}{J_0} )d \mu_{0}
\end{equation}
and
\begin{equation} \label{bbo1} \frac{d}{d \lambda}\int \log \mathfrak{J}_\lambda \, d \mu_{\mathfrak{J}_\lambda}|_{\lambda =1}= \int \log \tilde{J}_1 ( \frac{\tilde{J}_1-J_0}{\tilde{J}_1} )d \mu_{\tilde{J}_1}
\end{equation}
\end{proposition}

From \eqref{b1} and \eqref{bo1} it follows at once
\begin{proposition} \label{erte8812}
\begin{equation} \label{a83469}\frac{d}{d \lambda}D_{KL}(\mu_{\mathfrak{J}_\lambda}, \mu_1)|_{\lambda=0}= \int (\log J_0- \log J_1) \,( \frac{\tilde{J}_1-J_0}{J_0} )d \mu_{0}
\end{equation}
and
\begin{equation} \label{aaa1}\frac{d}{d \lambda}D_{KL}(\mu_{\mathfrak{J}_\lambda}, \mu_1)|_{\lambda=1}= \int (\log \tilde{J}_1- \log J_1) \,( \frac{\tilde{J}_1-J_0}{\tilde{J}_1} )d \mu_{\tilde{J}_1}
\end{equation}
\end{proposition}

From convexity we get that $\frac{d}{d \lambda}D_{KL}(\mu_{\mathfrak{J}_\lambda}, \mu_1)|_{\lambda=0}\leq \frac{d}{d \lambda}D_{KL}(\mu_{\mathfrak{J}_\lambda}, \mu_1)|_{\lambda=1}.$

\medskip

In the case $\frac{d}{d \lambda}D_{KL}(\mu_{\mathfrak{J}_\lambda}, \mu_1)|_{\lambda=0}>0$ we get that
\begin{equation} \label{aba83469} \int (\log J_0- \log J_1) \,( \frac{\tilde{J}_1-J_0}{J_0} )d \mu_{0}\geq 0.
\end{equation}

Note that from the inequality $ \frac{1}{x} -1\leq -\log x$ we get from above that
$$ \frac{d}{d \lambda}D_{KL}(\mu_{\mathfrak{J}_\lambda}, \mu_1)|_{\lambda=0}= \int \log J_0\, \frac{\tilde{J}_1-J_0}{J_0} d\mu_0 - \int \log J_1 \,\frac{\tilde{J}_1-J_0}{J_0} d \mu_{0}\leq$$
\begin{equation} \label{a8359}
\int \log J_0\, \frac{\tilde{J}_1-J_0}{J_0} d\mu_0 - \int \log J_1 \,\log \frac{J_0}{\tilde{J}_1} \, d \mu_{0}.
\end{equation}

\begin{example} \label{MaMa} We will present an example where the analogous result to Theorem 11.6.1 in \cite{Cover}
is not true.

Consider the shift invariant Markov probabilities $\mu_{J_j}=\mu_j, j=0,1,2,$ associated to the line stochastic matrices
\begin{equation} \label{feij20}P_j= \left(
\begin{array}{cc}
P^{1 1}_j & P^{1 2}_j\\
P^{2 1}_j & P^{2 2}_j
\end{array}
\right).
\end{equation}

Using the notation we considered before the $P_j$ is associated to $J_j$, when $j=0,1$, and $P_2$ is associated to $\tilde{J}_1.$ (and therefore $\mu_2=\mu_{\tilde{J}_1}).$

\begin{remark} \label{irri} Given a fixed $j$, when $P^{1 1}_j = P^{21}_j$ and $P^{1 2}_j= P^{2 2}_j,$
we get the i.i.d Bernoulli process with probabilities $(P^{1 1}_j, P^{12}_j)$.
\end{remark}

In this case the Jacobian $J_j$ is constant in the cylinder $\overline{r,s}, r,s=1,2,$ and
takes the value $P^{s r}_j$. The initial vector of probability is
$$\pi_j=(\pi^1_j,\pi^2_j)=(\frac{-1 +P^{2 2}_j }{-2 + P^{1 1}_j +P^{2 2}_j}, \frac{-1 +P^{1 1}_j }{-2 + P^{1 1}_j +P^{2 2}_j}) .$$

The entropy of $\mu_j, j=0,1,2$ is
$$ -\sum_{r, s } \pi_j^r \, P^{r s }_j \log P^{r s }_j= -\sum_{r, s } \pi_j^r \, P^{r s }_j \log P^{ s r }_j =- \int \log J_j d \mu_j. $$

For different choices of $P^{1 1}_0 , P^{2 2}_0, P^{1 1}_1, P^{2 2}_1, P^{1 1}_2 , P^{2 2}_2$ the value
$$\frac{d}{d \lambda}D_{KL}(\mu_{\mathfrak{J}_\lambda}, \mu_1)|_{\lambda=0}$$
can be positive or negative. In this way the Second Law regime and the fluctuation regime can occur for these triples.

Taking $P^{1 1}_0=0.2 , P^{2 2}_0=0.2 , P^{1 1}_1= 0.15, P^{2 2}_1=0.92, P^{1 1}_2=0.9 , P^{2 2}_2=0.12$, we get that $ \frac{d}{d \lambda}D_{KL}(\mu_{\mathfrak{J}_\lambda}, \mu_1)|_{\lambda=0}=0.362455>0$, but
$$ \frac{d}{d \lambda}D_{KL}(\mu_{\mathfrak{J}_\lambda}, \mu_1)|_{\lambda=0}= 0.2750>0$$
and
$$\int (\log J_1 - \log J_0) d \mu_0 + \int (\log J_0 - \log J_1) d \mu_{\tilde{J}_1}=-0.3578<0.$$

There are values of $P^{1 1}_0 , P^{2 2}_0, P^{1 1}_1, P^{2 2}_1, P^{1 1}_2 , P^{2 2}_2$ such that
$ \frac{d}{d \lambda}D_{KL}(\mu_{\mathfrak{J}_\lambda}, \mu_1)|_{\lambda=0}>0$ and $ \int (\log J_1 - \log J_0) d \mu_0 + \int (\log J_0 - \log J_1) d \mu_{\tilde{J}_1}>0.$





If we assume that all three probabilities  are i.i.d Bernoulli (see Remark \ref{irri}) we get that
$$\frac{d}{d \lambda}D_{KL}(\mu_{\mathfrak{J}_\lambda}, \mu_1)|_{\lambda=0}=\int (\log J_1 - \log J_0) d \mu_0 + \int (\log J_0 - \log J_1) d \mu_{\tilde{J}_1}= $$
$$ D_{K L} (\mu_2,\mu_0) - ( D_{K L} (\mu_{\tilde{J}_1},\mu_0) + D_{K L} (\mu_0,\mu_1))=$$
\begin{equation} \label{aba83469} (P^{1 1}_0 - P^{1 1}_2) (\log[1 - P^{1 1}_0 ] - \log[P^{1 1}_0 ] - \log[1 - P^{1 1}_1 ] + \log[P^{1 1}_1 ]).
\end{equation}

The above expression shows why Theorem 16.6.1 in \cite{Cover} is true but the analogous results are not true in the dynamical setting.
\end{example}

\medskip

\subsubsection{First problem} \label{tre}

Assume that $J_1\notin \Theta_1$ (associated to $\mu_1$) and $\mu_0$ (associated to $\tilde{J}=J_0\in \Theta_1$) satisfy
\begin{equation} \label{utro227} D_{KL}(\mu_1, \mu_0)= \min_{\tilde{J}\in \Theta_1} D_{KL}( \mu_1, \mu_{\tilde{J}}).
\end{equation}
Consider the $\mathfrak{J}_\lambda\in \Theta$, $\lambda\in [0,1]$ such that
$$ \mathfrak{J}_\lambda = \lambda \tilde{J}_1 + (1- \lambda) J_0,$$
where $\tilde{J}_1\neq J_0$.

We denote by $\mu_{\mathfrak{J}_\lambda}$ the equilibrium probability associated to the Jacobian $\mathfrak{J}_\lambda$

\begin{proposition} \label{erte2377}
$$\frac{d}{d \lambda}|_{\lambda=0}D_{KL}(\mu_1, \mu_{\mathfrak{J}_\lambda})=\int (1- \frac{\tilde{J}_1}{ J_0}) d \mu_1.$$

\end{proposition}

\begin{proof}
Denote $D_\lambda = D_{KL}(\mu_1, \mu_{\mathfrak{J}_\lambda})$. Then,
$$\frac{d \,D_\lambda}{d \lambda}|_{\lambda=0} =- \frac{d}{d \lambda} \int \log \mathfrak{J}_\lambda d \mu_1|_{\lambda=0} = -[\int \frac{d}{d \lambda} \log ( \tilde{J}_1 \lambda + (1- \lambda) J_0) d \mu_1|_{\lambda=0} ]=$$
$$ -\int \frac{\tilde{J}_1 - J_0}{ \mathfrak{J}_\lambda} d \mu_1|_{\lambda=0} = \int (1- \frac{\tilde{J}_1}{ J_0}) d \mu_1.$$

\end{proof}
\medskip

\begin{proposition} \label{erte23} Suppose the type-2 Pythagorean inequality is true
$$ D_{KL} (\mu_1 , \mu_{\tilde{J}_1}) \geq D_{KL} (\mu_1,\mu_0)+ D_{KL} (\mu_0 , \mu_{\tilde{J}_1}).$$
Then,
$$\frac{d}{d \lambda}|_{\lambda=0}D_{KL}(\mu_1, \mu_{\mathfrak{J}_\lambda})\geq D_{KL} (\mu_0 , \mu_{\tilde{J}_1})>0.$$

\end{proposition}
\begin{proof}
As $1 - \frac{1}{x} \geq \log x$, we get from  Proposition \ref{erte2377}
$$ \frac{d \,D_\lambda}{d \lambda}|_{\lambda=0} =\int (1- \frac{\tilde{J}_1}{ J_0}) d \mu_1 \geq \int (\log J_0 - \log \tilde{J}_1) d \mu_1=$$
$$ \int (\log J_1 - \log \tilde{J}_1) d \mu_1 - \int( \log J_1 - \log J_0) d \mu_1 =$$
$$ D_{KL} (\mu_1, \mu_{\tilde{J}_1}) -D_{KL} (\mu_1, \mu_0)\geq D_{KL} (\mu_0 , \mu_{\tilde{J}_1})>0 .$$

\end{proof}

\begin{proposition} \label{erte2333} Suppose
$$\frac{d}{d \lambda}|_{\lambda=0}D_{KL}(\mu_1, \mu_{\mathfrak{J}_\lambda})< 0,$$
then
is true the type-2 triangle inequality
\begin{equation} \label{oop} D_{KL} (\mu_1 , \mu_{\tilde{J}_1})< D_{KL} (\mu_1 , \mu_0) +D_{KL} (\mu_0,\mu_{\tilde{J}_1}).
\end{equation}
\end{proposition}

\begin{proof}
Note that
$$ D_{KL} (\mu_1, \mu_{\tilde{J}_1}) -D_{KL} (\mu_1, \mu_0)=$$
\begin{equation} \label{oops1} \int (\log J_0 - \log \tilde{J}_1) d \mu_1 \leq\int (1- \frac{\tilde{J}_1}{ J_0}) d \mu_1 = \frac{d}{d \lambda}|_{\lambda=0}D_{KL}(\mu_1, \mu_{\mathfrak{J}_\lambda})< 0,
\end{equation}

Then
$$ D_{KL} (\mu_1 , \mu_{\tilde{J}_1})< D_{KL} (\mu_1 , \mu_0) < D_{KL} (\mu_1 , \mu_0) +D_{KL} (\mu_0,\mu_{\tilde{J}_1}).$$

The above is equivalent to
$$ D_{KL} (\mu_1 , \mu_{\tilde{J}_1}) - D_{KL} (\mu_1 , \mu_0)> D_{KL} (\mu_0,\mu_{\tilde{J}_1})>0,$$
which is a contradiction to \eqref{oops1}.
\medskip

\end{proof}

\section{Appendix - On Fourier-like Hilbert basis}  \label{marbas}

Consider $M=\{0,1\}^\mathbb{N}$ and a equilibrium probability $\mu_A=\mu_{\log J}$ on  $\{0,1\}^\mathbb{N}$ (notation of \cite{GKLM}), associated to a H\"older potential $A=\log J$, where $J$ is a Jacobian.

Given a finite word $x =(x_1,x_2,...,x_k)\in \{0,1\}^k$, $k \in \mathbb{N}$, we denote by $[x]=[x_1,x_2,...,x_k]$ the associated cylinder set in
$\Omega=\{0,1\}^\mathbb{N}$.

For each $n$ denote by $\mathfrak{C}_n$ the set of all cylinders $[x]$ of length $n$ which is a partition of $M$.
The  lexicographic order $\preceq$
on $M=\{0,1\}^\mathbb{N}$ makes it  a totally ordered set.

Denote by $\hat{\mathfrak{S}}: M \to M$
a  function such that for any $x \in M$, we get $\hat{\mathfrak{S}} (1,x)=(0,x),$ and moreover, $\hat{\mathfrak{S}} (0,x)=(1,x).$ We denote by $\mathfrak{S}=\hat{\mathfrak{S}}|_{[1]}.$

Note that a function $\varphi$ in the kernel of the Ruelle operator $\mu_J$ is determined by its values on the cylinder $[0]$.
Indeed, if $\varphi$ is in the kernel, we get that for all $x$
$$ \varphi(1,x)= - \frac{J(0,x) \, \varphi(0,x) }{J(1,x)}.$$

This is equivalent to say that $\varphi$ can be expressed as
\begin{equation} \label{ere1} \varphi= \varphi \,\, \mathfrak{1}_{[0]} -   \frac{(J \circ \mathfrak{S} )\, (\varphi \circ \mathfrak{S} )}{ J \,}\, \, \mathfrak{1}_{[1]}.
\end{equation}

Initially, we will present a simple example of Fourier-like basis which will help to understand more general cases which will be addressed later.

\begin{example} \label{exx1} In this example $\mu$ is the probability of maximal entropy which is associated to the potential $-\log 2$.  In this case the functions
$\varphi$ {in  the kernel of} the Ruelle operator can be expressed as
\begin{equation} \label{ere2} \varphi= \varphi \,\, \mathfrak{1}_{[0]} -   \, (\varphi  \circ\mathfrak{S} )\, \, \mathfrak{1}_{[1]}.
\end{equation}

First, we  will present a  natural  Fourier-like basis for $L^2(\mu)$ (and later for the kernel of the Ruelle operator).

We order the cylinder sets in $\mathfrak{C}_n$ using this order.
For example, when
$n=2$ we get
$$ (0,0), (0,1),(1,0),(1,1),$$
and $n=3$ we get
\begin{equation} \label{acima} (0,0,0), (0,0,1),(0,1,0),(0,1,1),(1,0,0),(1,0,1),(1,1,0),(1,1,1).
\end{equation}

By abuse of language we can say that $(0,1,1)\preceq (1,0,0)$ and also that
$(0,1,1)\preceq (1,1,0)$.

Given $\mathfrak{C}_n$, we say that the cylinder $ [x] \in\mathfrak{C}_n$
is odd (respectively, even) if occupies an odd (respectively, even) position in the above defined order of cylinders. For example, in \eqref{acima} the cylinders $(0,0,0)$ and $(0,1,0)$ are odd and $(0,0,1)$ and $(0,1,1)$ are even.

We say that the cylinder $[x] \in\mathfrak{C}_n$ has the cylinder $[y] \in\mathfrak{C}_n$ as its next neighborhood at the right side if  $[x] \preceq [y]$, and there is no cylinder $[z] \in \mathfrak{C}_n$, such that, $[x] \preceq [z] \preceq [y]$. In this case we say that $[x],[y]$ is a neighborhood
pair of cylinders.

We will define an orthonormal family $\mathfrak{F}=\{\alpha_m,\beta_n, m\geq 2, n \geq 1 \}$ of linear independent continuous functions in
$L^2 (\mu)$, which is uniformly bounded on $L^2 (\mu)$ (all elements have $L^2$ norm equal to $1$) and also on  $C^0$.

For a given $n\geq 2$ we consider the function
$\alpha_n$ which is constant in each cylinder $[x]=(x_1,x_2,...,x_n)$ of $\mathfrak{C}_n$,
taking the value $1$, if in the ordering of cylinders in $\mathfrak{C}_n$ the cylinder $[x]$
it occupies an even position, and taking the value $-1$, if in the ordering of cylinders in $\mathfrak{C}_n$
it occupies an odd position.

For example,
$\alpha_2 = \mathfrak{1}_{(0,0)} - \mathfrak{1}_{(0,1)} + \mathfrak{1}_{(1,0)} - \mathfrak{1}_{(1,1)}.$

The functions $\alpha_n$ have $L^2$ norm equal to $1$. It is easy to see that
$<\alpha_n,\alpha_m>=0$, $n \neq m$, $m,n\geq 2$. It follows from \eqref{ere2} that
the functions $\alpha_n$, with $n\geq 2$, are orthogonal to the
kernel of $\op{L}_A$.

In a  little different procedure,
for a given $n\geq 2$ we consider the function
$\beta_n$ which is constant in  cylinders $[x]$ in $\mathfrak{C}_{n}$  in the following way: in the cylinder $[0]$ we define $\beta_n=\alpha_n$, for all $n$.  For $y$ on the cylinder $[1]$ we define $\beta_n(y)= -\alpha_n ( \mathfrak{S}(y)).$

For example, $\beta_2= \mathfrak{1}_{(0,0)} - \mathfrak{1}_{(0,1)} - \mathfrak{1}_{(1,0)} + \mathfrak{1}_{(1,1)}$ and
$$\beta_3 = \mathfrak{1}_{(0,0,0)} - \mathfrak{1}_{(0,0,1)} + \mathfrak{1}_{(0,1,0)} - \mathfrak{1}_{(0,1,1)}- \mathfrak{1}_{(1,0,0)} + \mathfrak{1}_{(1,0,1)} - \mathfrak{1}_{(1,1,0)} + \mathfrak{1}_{(1,1,1)} .$$

We define  $\beta_1= \mathfrak{1}_{(0)} - \mathfrak{1}_{(1)}$.

The functions $\beta_n$   are in the kernel of the Ruelle operator for the potential $-\log 2.$

The functions $\beta_n$ have $L^2$ norm equal to $1$. One can show  that
$<\beta_n,\beta_m>=0$, $n \neq m$, $m\geq 2, n\geq 1$. Moreover,
$<\alpha_n,\beta_m>=0,$ for all $m,n$.

The functions $\alpha_m$ and $\beta_n$, $m\geq 2,n \geq 1$, are H\"older continuous for the usual metric on $M=\{0,1\}^\mathbb{N}$.
The
family $\mathfrak{F}$ is the union of all functions $\alpha_n$ and $\beta_n$, $m\geq 2,n \geq 1$.

\begin{remark} \label{pois1}
One can show that the sigma algebra generated by the functions in $\mathfrak{F}$ is the Borel sigma-algebra in $M$. Indeed, one can get  any cylinder set on $\{0,1\}^\mathbb{N}$ as intersection of preimages of open sets for
a finite number of functions in $\mathcal{F}$. In order to illustrate this fact note that the cylinder $[0,0,0]$ can be obtained as
$$[0,0,0] = \alpha_3^{-1}(0,2) \cap \beta_3^{-1}(0,2) \cap \alpha_2^{-1}(0,2).$$
\end{remark}
It follows that  $\mathfrak{F}=\{\alpha_m,\beta_n, n\geq 2, m \geq 1 \}$ is an orthonormal basis for $L^2(\mu)$, where $\mu$ is the measure of maximal entropy.

\begin{remark} \label{pois2}
Using a similar reasoning one can show that the family
$\mathfrak{F}_0=\{\beta_m, m \geq 1 \}$ is an orthonormal basis for the kernel of the Ruelle operator
$\op{L}_{-\log 2}.$ The family $\mathfrak{F}_0$ is a Fourier-like family.
\end{remark}

\medskip

$\,\,\,\,\,\,\,\,\,\,\,\,\,\,\,\,\,\,\, \,\,\,\,\,\,\,\,\,\,\,\,\,\,\,\,\,\,\,\,\,\,\,\,\,\,\,\,\,\,\,\,\,\,\,\,\,\, \,\,\,\,\,\,\,\,\,\,\,\,\,\,\,\,\,\,\, \,\,\,\,\,\,\,\,\,\,\,\,\,\,\,\,\,\,\,\,\,\,\,\,\,\,\,\,\,\,\,\,\,\,\,\,\,\,\,\,\,\,\,\,\,\,\,\,\,\,\,\,\,\,\,\,\, \,\,\,\,\,\,\,\,\,\,\,\,\,\,\,\,\,\,\,\,\,\,\,\,\,\,\,\,\,\,\,\,\,\,\,\,\,\,\,\,\,\,\,\,\,\,\,\,\,\,\,\,\,\,\,\,\, \,\,\,\,\,\,\,\,\,\,\,\,\,\,\diamond$

\end{example}

\subsection{A Fourier-like  basis for the kernel in the case of Markov probabilities}  \label{marbas1}

Consider $M=\{0,1\}^\mathbb{N}$ and denote by $K$  the set of stationary Markov probabilities taking values in $\{0,1\}$. In this section, we will present explicit expressions for a Fourier-like  basis of the kernel of the associated Ruelle operator. The functions on the basis are constant in cylinders.

Consider a  shift invariant
Markov probability $\mu$ obtained from a  row stochastic matrix $(P_{i,j})_{i,j=0,1}$ with positive entries and the initial left invariant  vector of probability $\pi=(\pi_0,\pi_1)\in \mathbb{R}^2$.  We denote by $A$ the H\"older potential associated to such probability  $\mu$ (see Example 6 in \cite{LR}). There exists an explicit countable orthonormal basis $\hat{a}_x$, indexed by finite words $[x]$, for the set of H\"older functions  in the kernel of the Ruelle operator $\op{L}_A$  (see \cite{LR1} or the paragraph after expression \eqref{tororo15}).

Given $r \in (0,1)$ and $s\in (0,1)$ we denote
 \begin{equation} \label{tororo2} P= \left(
\begin{array}{cc}
P_{0,0} & P_{0,1}\\
P_{1,0} & P_{1,1}
\end{array}\right)=   \left(
\begin{array}{cc}
r & 1-r\\
1-s  & s
\end{array}\right) .
\end{equation}

In this way $(r,s)\in (0,1) \times (0,1)$ parameterize all {\bf row} stochastic matrices we are interested.

The explicit expression for $\mu$ is
\begin{equation} \label{ppp} \mu [x_1,x_2,..,x_n] = \pi_{x_1} \,P_{x_1,x_2}\,P_{x_2,x_3}\,...\, P_{x_{n-1},x_n}.
\end{equation}

Recall that in the Markov case  the family of H\"older functions
\begin{align}
   e_{[x]} =\frac{1}{\sqrt{\mu([x])}} \sqrt{\frac{P_{x_n,1}}{P_{x_n,0}\,}} \, \mathfrak{1}_{[x0]} - \frac{1}{\sqrt{\mu([x])}} \sqrt{\frac{P_{x_n,0}}{P_{x_n,1} }} \, \mathfrak{1}_{[x1]},
    \label{eq52}
    \end{align}
where $x$ is a finite word  is an orthonormal (Haar) Hilbert basis for $\mathcal{L}^2 (\mu)$ (see \cite{KS} for a general expression  and \cite{CHLS}  for the above one). The integral of the functions $e_x$ is equal to zero. To be more precise we need to add to this family the functions $\mathfrak{1}_{[0]}$ and $\mathfrak{1}_{[1]}$ in order to have a basis.

\begin{theorem}  \label{kert} For any two by two stochastic  matrix $P$ there exists an orthogonal basis of the kernel of  the Ruelle operator
$\op{L}_A$, denoted by
$\mathcal{B}=\{\gamma_n, n \in \mathbb{N}\}$, and constants $\alpha>0, \beta>0$, such that

I) the  functions $\gamma_n$, $n \in \mathbb{N}$, in   the family $ \mathcal{B} $ have $C^0$ and $L^2(\mu_A)$ norms uniformly bounded above by the constant $\beta>0$,

II) the  functions $\gamma_n$, $n \in \mathbb{N}$, in   the family $ \mathcal{B} $ have  $C^0$ and $L^2(\mu_A)$ norms uniformly bounded below by the constant $\alpha>0$,

\end{theorem}

\begin{proof}

From \cite{LR} it is known that
for each finite word $(x_1,x_2,..,x_n)$, the function
\begin{eqnarray*}
 a_x & =& \frac{ \sqrt{\pi_{x_1}  }} {\sqrt{\pi_{0}}\sqrt{P_{0,x_1} }} \,\, e_{[0,x_1, x_2,..,x_n]} - \,\frac{ \sqrt{\pi_{x_1}  }} {\sqrt{\pi_{1}}\sqrt{P_{1,x_1} }}\,\,e_{[1,x_1, x_2,..,x_n]}\\
& = &  \frac{ \sqrt{\pi_{x_1}  }} {\sqrt{\pi_{0}}\sqrt{P_{0,x_1} }} \,   \frac{1}{\sqrt{\mu([0 x])}}\,[ \sqrt{\frac{P_{x_n,1}}{P_{x_n,0}\,}} \, \mathfrak{1}_{[0x0]} -  \sqrt{\frac{P_{x_n,0}}{P_{x_n,1} }} \, \mathfrak{1}_{[0x1]} ]
\end{eqnarray*}
 \begin{equation} \label{tororo15}  -\,\frac{ \sqrt{\pi_{x_1}  }} {\sqrt{\pi_{1}}\sqrt{P_{1,x_1} }} \,\frac{1}{\sqrt{\mu([1 x])}} \,[\,\sqrt{\frac{P_{x_n,1}}{P_{x_n,0}\,}} \, \mathfrak{1}_{[1x0]} -  \sqrt{\frac{P_{x_n,0}}{P_{x_n,1} }} \, \mathfrak{1}_{[1x1]}  ]. \end{equation}
is H\"older and in the kernel of the Ruelle operator. When $x$ ranges in the set of finite words we get that $\hat{a}_x=\frac{a_x}{|a_x|}$ is an orthonormal Haar basis for the H\"older functions on the kernel of the Ruelle operator $\op{L}_A$ associated to $\mu$ (see \cite{LR1}). In order to be more precise we need to add two more functions to the family to get a basis (see \cite{LR1}).

The $L^2$ norm of $a_x$ does not depend on the finite word $x$. Note that this family is not Fourier-like because
the $C^0$ norm  of $a_x$ is not uniformly bounded when $x$ ranges in the set of all finite words.

Note that the values  $\frac{ \sqrt{\pi_{j}  }} {\sqrt{\pi_{0}}\sqrt{P_{0,j} }}, \frac{ \sqrt{\pi_{j}  }} {\sqrt{\pi_{1}}\sqrt{P_{1,j} }}, \frac{\sqrt{P_{i,j}}}{\sqrt{P_{m,n}}\,}$, $i,j,m,n=0,1$, are bounded above by a constant $\beta>0$ and below by a constant $\alpha>0$.

We denote by $V$ the subspace of $L^2(\mu)$ generated by the span of the functions $a_x$, where $x$ is a finite word. If $\varphi$ in $V$, then the extension of $\op{L}_A$ to $V$ is such that $\op{L}_A(\varphi)=0.$

Denote by $\mathfrak{C}_n$ the set of all cylinders of length $n$ which is a partition of $M$. The sets of the form $[0x0], [0x1],[1x0],[1x1],$ where
$x$ ranges in $\mathfrak{C}_n$, is also a partition of $M$ (defines the set
$\mathfrak{C}_{n+2}$).

Note that for a fixed $n$ and a fixed cylinder $x=(x_1,x_2,..,x_n)\in \mathfrak{C}_n$
$$\frac{\sqrt{\mu (x 0)}}{\sqrt{\mu (x 1)}}= \frac{\sqrt{\pi_{x_1}\, P_{x_1,x_2 } \,...P_{x_{n-1},x_n }  P_{x_{n},0 }}  }{\sqrt{ \pi_{x_1}\, P_{x_1,x_2 } \,...P_{x_{n-1},x_n } P_{x_{n},1 }}  }=
\frac{\, \sqrt{P_{x_n,0 }}\,  }{\sqrt{ \,P_{x_n,1 }\,}  }
$$
and, for fixed $\mu$, this value is bounded above and below by a bound which is independent of $n$ and the cylinder  $x$.

Given $x$ set $b_x$ as
$$\sqrt{\mu (x 0)} a_x=b_x.$$

The function $b_x$ is continuous and  uniformly bounded in the $C^0$ norm. The $L^2$ norm and also the $C^0$ norm of $b_x$ are uniformly bounded above by a constant $\beta>0$, when $x$ ranges in the set of all cylinders.  Note also that the values of the norm $|b_x(y)|$, $y\in M$,  are uniformly bounded below by a constant $\alpha,$ independent of the finite word $x$.

In a more explicit form:
for $x=(x_1,x_2,..,x_n)$, we get
\begin{eqnarray*}
 b_x & =& \frac{ \sqrt{\pi_{x_1}  }} {\sqrt{\pi_{0}}\sqrt{P_{0,x_1} }} \,\, e_{[0,x_1, x_2,..,x_n]} - \,\frac{ \sqrt{\pi_{x_1}  }} {\sqrt{\pi_{1}}\sqrt{P_{1,x_1} }}\,\,e_{[1,x_1, x_2,..,x_n]}\\
& = &  \frac{ \sqrt{\pi_{x_1}  }} {\sqrt{\pi_{0}}\sqrt{P_{0,x_1} }} \,   \,[ \sqrt{\frac{P_{x_n,1}}{P_{x_n,0}\,}} \, \mathfrak{1}_{[0x0]} -  \sqrt{\frac{P_{x_n,0}}{P_{x_n,1} }} \, \mathfrak{1}_{[0x1]} ]
\end{eqnarray*}
 \begin{equation} \label{tororo1555}  -\,\frac{ \sqrt{\pi_{x_1}  }} {\sqrt{\pi_{1}}\sqrt{P_{1,x_1} }}  \frac{\, \sqrt{P_{x_n,0 }}\,  }{\sqrt{ \,P_{x_n,1 }\,}  }  \,[\,\sqrt{\frac{P_{x_n,1}}{P_{x_n,0}\,}} \, \mathfrak{1}_{[1x0]} -  \sqrt{\frac{P_{x_n,0}}{P_{x_n,1} }} \, \mathfrak{1}_{[1x1]}  ]. \end{equation}

Given $x$, the possible values attained by $b_x$ in the cylinders $[0x0],[0x1],[1x0],[1x1]$ are in a finite set, when $[x]$ ranges in the set of all possible cylinders with different sizes. Note that for a  general stochastic matrix $P$  some of these
possible values can coincide (but not for a generic - on the parameters $(r,s)\in(0,1)\times(0,1)$ - matrix $P$). But his will not be a problem.

Given $n$, note that the support of the functions $b_x$ are all disjoint
when $[x]$ ranges in the set $ \mathfrak{C}_n$. The union of the supports is the set $M$. Remember that  for a fixed $n$, when $x$ ranges in the set of all words of length $n$, the cylinders $[i,x,j]$, $i,j=0,1$, determine the  partition  $ \mathfrak{C}_{n+2}$ of $M$.

For each $n$ we are going to define a function
$\gamma_n$ of the form
$\gamma_n= \sum_{x\in \mathfrak{C}_n} s_x\,  b_x,$ where  all $s_x>0$ are close to $1$ in such way that the function  $\sum_{x\in \mathfrak{C}_n} s_x\,  b_x$ has $C^0$ and $L^2$ norm smaller than $\beta>0$, and $L^2$ norm larger than $\alpha>0$. Moreover,  we get that $|s_x b_x|>\alpha$, for all word $x$ of any length.

For each $n$, we set $\gamma_n$ as  the continuous function $\gamma_n= \sum_{x\in \mathfrak{C}_n} s_x\,  b_x.$  The family
$\gamma_n$, $n \in \mathbb{N}$, is  orthogonal, $C^0$ and $L^2(\mu_A)$ uniformly bounded, but not orthonormal. Dividing by the norm we get an orthonormal family $\mathfrak{F}_0$ (and for simplification we will also denote its elements by $\gamma_n$)

We claim that  the sigma algebra generated by $\gamma_n$, $n \in \mathbb{N},$ contains all cylinders of  all sizes. This claim can be obtained  from a tedious procedure  following the reasoning  of Remark \ref{pois2} of Example \ref{exx1} and is left for the reader.

As  the sigma-algebra generated by all $\gamma_n$, $n \in \mathbb{N}$, is the Borel sigma algebra, the span of the family of all $\gamma_n$,  $n \in \mathbb{N}$, contains the set of H\"older functions on the kernel of the Ruelle operator. From this follows that $\mathfrak{F}_0$ is an orthonormal basis of the kernel of the Ruelle operator in the case of Markov probabilities.
$\mathfrak{F}_0$ is a Fourier-like basis.
\end{proof}

\subsection{A Fourier-like basis for the space $L^2(\mu)$  in the general case of equilibrium probabilities on $\{0,1\}^\mathbb{N}$} \label{ger}

We point out that a similar (but more complex) procedure as in Theorem \ref{kert} allows one to
get a family $\rho_n, n \in \mathbb{N},$ which is a Fourier-like basis of the space $L^2(\mu)$, where $\mu$ is a equilibrium  probability on $\{0,1\}^\mathbb{N}$ for a H\"older potential $A=\log J$.

It follows from \cite{CHLS} (using results from \cite{KS}) that
the family of H\"older functions (called Haar family)
    \begin{equation} \label{eert1}
    e_x  = \sqrt{\frac{\mu([x1])}{\mu([x0]) \, \mu([x])}} \, 1_{[x0]} - \sqrt{\frac{\mu([x0])}{\mu([x1]) \, \mu([x])}} \, 1_{[x1]}  ,
 \end{equation}

 when $x$ ranges in the set of all finite words with letters in $\{0,1\}$, is an orthonormal basis of $L^2(\mu)$. But this  basis is not Fourier-like (it is not $C^0$ uniformly bounded).

We will produce a Fourier-like basis from this Haar basis.

 Bowen formula for a equilibrium probability  $\mu$ with Jacobian $J$  (see  Definition 1.1 in \cite{IoYa} o in \cite{Bow})  claims that there exists $K_1,K_2>0$, such that  for all $n$, all cylinder $x=[x_1,x_2,..,x_n]$ and any $y\in [x_1,x_2,..,x_n]$
   \begin{equation} \label{iommi} K_1<\frac{\mu([x_1,x_2,..,x_n])}{\Pi_{j=0}^{n-1} J(\sigma^j(y))}< K_2.
   \end{equation}

    Given $n$, when $x$ ranges in $\mathfrak{C}(n)$, the cylinders $[x,0]$ and $[x,1]$ determine the partition $\mathfrak{C}(n+1)$.

     We claim that the quotients
    \begin{equation} \label{iommi8} \frac{\mu([x 0])}{\mu([x 1])},
    \end{equation} when $x$ ranges in $\mathfrak{C}(n)$, are bounded above and below by positive constants which are independent of $n$.

For the proof of the claim,   for a fixed $x\in \mathfrak{C}(n)$, take initially $y_{[x,r]}\in [x_1,x_2,..,x_n,r]$, $r=0,1$, and it follows  from \eqref{iommi} that
 \begin{equation} \label{iommi2} K_1<\frac{\mu([x_1,x_2,..,x_n,r])}{\Pi_{j=0}^{n} J(\sigma^j(y_{[x,r]}))}< K_2.
   \end{equation}

   Note that $ y_{[x,r]}\in[x]$, for $r=0,1$.

  There exists $C_1,C_2>0$, such that
  \begin{equation} \label{iommi7}  C_1< \frac{ \Pi_{j=0}^{n} J(\sigma^j(y_{[x,0]})) }{\Pi_{j=0}^{n} J(\sigma^j(y_{[x,1]})) }<C_2.
  \end{equation}

   Indeed,  from \eqref{iommi} applied  for the case  $[x]=[x_1,x_2,..,x_n])$,  and $r=0,1$, we get
   \begin{equation} \label{iommi1} K_1<\frac{\mu([x_1,x_2,..,x_n])}{\Pi_{j=0}^{n-1} J(\sigma^j(y_{[x,r]}))}< K_2.
   \end{equation}

   Note also that $\Pi_{j=0}^{n} J(\sigma^j(y_{[x,r]}))= \Pi_{j=0}^{n-1} J(\sigma^j(y_{[x,r]}))\, J(\sigma^n(y_{[x,r]}))$, for $r=0,1$.

   Therefore,
 $$   \frac{ \Pi_{j=0}^{n} J(\sigma^j(y_{[x,0]})) \,}{\Pi_{j=0}^{n} J(\sigma^j(y_{[x,1]})) }= \frac{ \Pi_{j=0}^{n-1} J(\sigma^j(y_{[x,0]}))\, J(\sigma^n(y_{[x,0]}))}{\Pi_{j=0}^{n-1} J(\sigma^j(y_{[x,1]})) \,  J(\sigma^n(y_{[x,1]}))}<$$
 $$\frac{K_2}{ K_1} \frac{ \mu([x_1,x_2,..,x_n])}{\mu([x_1,x_2,..,x_n])} \frac{ \, J(\sigma^n(y_{[x,0]}))}{ J(\sigma^n(y_{[x,1]}))}=\frac{K_2}{ K_1} \,\frac{ \, J(\sigma^n(y_{[x,0]}))}{ J(\sigma^n(y_{[x,1]}))} ,$$
 and this shows the existence of $C_2>0$ in  the last inequality in \eqref{iommi7}. The proof of the other inequality in \eqref{iommi7} is similar. Then, \eqref{iommi7} is true.

   Finally, from \eqref{iommi} and \eqref{iommi7}
   \begin{equation} \label{iommi9} \frac{\mu([x_1,x_2,..,x_n,0])}{\mu([x_1,x_2,..,x_n,1])}< \frac{ K_1\, K_2\,\Pi_{j=0}^{n} J(\sigma^j(y_{[x,0]})) }{\Pi_{j=0}^{n} J(\sigma^j(y_{[x,1]})) } <K_1\,K_2\, C_2
   \end{equation}

   The proof for the lower bound in \eqref{iommi8} is similar showing that
   \eqref{iommi8} is true.

   \medskip

   Now we are going to define for each $n$ a function $\rho_n$ whose support is the set $M$.

   For fixed $n$ and each finite word $x\in \mathfrak{C}(n)$ take
    \begin{equation} \label{eert12}
 c_{[x]} = \sqrt{\mu([x])}  \,e_x  = \sqrt{\frac{\mu([x1])}{\mu([x0]) \,}} \, 1_{[x0]} - \sqrt{\frac{\mu([x0])}{\mu([x1]) }} \, 1_{[x1]}  ,
 \end{equation}

When $x$ ranges in the set of all words of length $n$, the cylinders $[x,j]$, $j=0,1$, determine the  partition  $ \mathfrak{C}_{n+1}$ of $M$.

For each $n$, set $\rho_n$ as  the continuous function
\begin{equation} \label{kret} \rho_n = \sum_{x \in \mathfrak{C}_{n}} c_{[x]}\end{equation}

From the above and \eqref{iommi8} is easy to see that the family $\rho_n $,  $n \in \mathbb{N},$ is an orthogonal family for $L^2(\mu)$, which is uniformly  $C^0$ and $L^2$ bounded above  and bounded away from zero.

In order to show that is a basis it is necessary to show that the family $\rho_n $,  $n \in \mathbb{N},$ generate the Borel sigma-algebra on $M$.  This can be achieved  following the same line of the reasoning of Remark \ref{pois1} in Example \ref{exx1}.

Therefore, the family $\rho_n $,  $n \in \mathbb{N},$ is a
Fourier-like basis for the space $L^2(\mu)$, where $\mu$ is the equilibrium state for the H\"older potential  $A=\log J.$

\subsection{A Fourier-like basis for the kernel of the Ruelle operator in the general case of equilibrium probabilities on $\{0,1\}^\mathbb{N}$.} \label{aqui1}

In this section we will exhibit a family $\hat{\rho}_n$ which is a Fourier-like basis for the kernel of the Ruelle operator $\mu_J.$

For the general case, a typical function $\psi$ on the kernel of $\op{L}_A$, where $A=\log J,$ can be obtained in the following way: take a H\"older continuous function $\varphi$ and consider first
$\psi$ defined on the cylinder $[0]$, where  we set $\psi (0,x)=\varphi (0,x)  $, which therefore it is well defined for all $x\in M$. On the other hand, on the cylinder $[1]$ we define $\psi: [1]\to \mathbb{R}$ in { such way that for  $y=(1,x) \in [1]$
$$\psi(y)= - \frac{(J \circ \mathfrak{S}) (y) \, (\varphi \circ \mathfrak{S}) (y) }{ J (y)\,}\, \, \mathfrak{1}_{[1]}(y) .$$

It is easy to see that $\psi$ is on the kernel of $\op{L}_A$.  Indeed, given $x$ we get
$$\op{L}_A (\psi) (x) =J(0,x) \psi (0,x) + J(1,x) \psi (1,x) =$$
$$J(0,x) \varphi(0,x) - J( 1,x)  \frac{J ( \mathfrak{S} (1,x))\, \varphi (\mathfrak{S} (1,x) )}{ J (1,x)}\, \, \mathfrak{1}_{[1]}(1,x)= $$
$$J(0,x) \varphi(0,x) -   J (0,x)\, \varphi (0,x) \, = 0.$$

The function $\mathcal{T}$ taking $\varphi:M \to \mathbb{R}$ to $\psi:M \to \mathbb{R}$  in the kernel is defined by
$$\psi=\mathcal{T}(\varphi)= \varphi \,\, \mathfrak{1}_{[0]} -   \frac{(J \circ \mathfrak{S} )\,( \varphi \circ \mathfrak{S} )}{ J \,}\, \, \mathfrak{1}_{[1]}$$

We claim that $\mathcal{T}$ is a linear projection  onto   the kernel, that is $\mathcal{T}(\varphi)=\varphi$, for all $\varphi$ on the kernel. Moreover, if $\varphi$ is on the kernel, we get that for all $x$
$$ \varphi(1,x)= - \frac{\varphi(0,x) J(0,x)}{J(1,x)}.$$

\begin{remark} \label{hjg} Note that  the function $    \frac{(J \circ \mathfrak{S} )\,( \varphi \circ \mathfrak{S} )}{ J \,}\, $ (which is defined on  $[1]$) is linear on $\varphi$ (a function defined just on $[0])$.
\end{remark}

We consider  the family   $\hat{\mathfrak{a}}_x:[0] \to \mathbb{R}$ (see\eqref{eert12}) where
 \begin{equation} \label{eert123}
 \hat{\mathfrak{a}}_x := c_{[0 x]}   = \sqrt{\frac{\mu([0x1])}{\mu([0x0]) \,}} \, 1_{[0x0]} - \sqrt{\frac{\mu([0x0])}{\mu([0x1]) }} \, 1_{[0x1]},
 \end{equation}
where $x$ ranges in the set of all finite words $x$. For fixed $n$, the pair of cylinders $[0x0]$, $[0x1]$, where $x$ ranges in $\mathfrak{C}_{n}$, describes a partition of cylinder $[0]$
by the cylinders in $\mathfrak{C}_{n+2}$  (using  just the ones contained in the cylinder $[0])$.

Any given H\"older function $f$ with support on $[0]$ can be written as an infinite sum $f=\sum_x r_x \hat{\mathfrak{a}}_x.$  Indeed, taking $f \mathfrak{1}_{[0]} + 0 \mathfrak{1}_{[1]}$, and expressing it in the basis
\eqref{eert12}, we will just need to take elements of the form $c_{[0x]}.$

For any finite word  $x$ consider the  function $\mathfrak{a}_x$ (defined on $M$) which in the cylinder $[0]$ coincides with $\hat{\mathfrak{a}}_x$, and in the cylinder $1$ the function $\mathfrak{a}_x$  is {\bf given by}
\begin{equation} \label{ggs} \mathfrak{a}_x=   -\,  \frac{J \circ \mathfrak{S} \,  }{ J \,}\, [\,\sqrt{\frac{\mu([0x1])}{\mu([0x0]) \,}} \, 1_{[1x0]} - \sqrt{\frac{\mu([0x0])}{\mu([0x1]) }} \, 1_{[1x1]} \,] \,.
\end{equation}

Each function $\mathfrak{a}_x$ is in the kernel of $\mu_J$. It follows from \eqref{iommi9} that  the functions $\mathfrak{a}_x$, where $x$ is a finite word,
are uniformly bounded below and above in the $C^0$ and $L^2$ norm.

Note that  $\sum_x r_x \hat{\mathfrak{a}}_x$ restricted to  the cylinder $[0]$ coincides
with  the $f$ given above.

It follows from the above and Remark \ref{hjg} that any function on the kernel can be written as a
infinite sum $\sum_x r_x \mathfrak{a}_x$.

We denote by $\tau_0:M \to [0]$ the inverse of $\sigma|_{[0]}$ and
$\tau_1:M \to [0]$ the inverse of $\sigma|_{[1]}$. The functions $\tau_0$ and $\tau_1$ are called the inverse branches of $\sigma$.

\begin{lemma} \label{enf3}  Given a function $\varphi:[0]\to \mathbb{R}$ we get that
\begin{equation} \label{poo1} \int_{[0]} \varphi d \mu = \int_M \varphi (\tau_0)\, J(\tau_0) d \mu.
\end{equation}
In a similar way,  function $\phi:[1]\to \mathbb{R}$ we get that
\begin{equation} \label{poo2} \int_{[1]} \phi d \mu = \int_M \phi (\tau_1)\, J(\tau_1) d \mu.
\end{equation}
\end{lemma}

The above Lemma which characterizes the Jacobian $J$ as a Radon-Nykodin derivative is a classical result in Thermodynamic formalism (see (5) in \cite{LR} or \cite{Ma}).

\begin{lemma} \label{enf2}   Given a  continuous  function $f:[0] \to \mathbb{R}$, then
\begin{equation} \label{kw1}  \int_{[0]} f d \mu = \int_{[1]} \frac{(J \circ\mathfrak{S}) (y) )\,  (f \circ\mathfrak{S})}{ J (y)\,} d\mu.\end{equation}
\end{lemma}

\begin{proof} First note that $\mathfrak{S}= \tau_0 \circ T.$

In \eqref{poo2} take
$$\phi=   \frac{(J \circ\mathfrak{S} )\,  (f \circ\mathfrak{S})}{ J (y)\,}=
 \frac{J ( (\tau_0 \circ T) (y) )\,  f (( \tau_0 \circ T) (y) )}{ J (y)\,}.
$$

Then, from \eqref{poo2} and \eqref{poo1} we get
$$ \int_{[1]} \phi d \mu =\int_M  (\phi \circ \tau_1) \, (J \circ \tau_1) d \mu= $$
$$ \int_M \frac{J ( (\tau_0 \circ T) (\tau_1(y)) )\,  f (( \tau_0 \circ T) (\tau_1(y)) )}{ J (\tau_1(y))\,} \, (J \circ \tau_1) d \mu=$$
$$ \int_M J ( \tau_0 (y) )\,  f ( \tau_0 (y) ) d \mu=\int_{[0]}\, f \,d \mu.$$

\end{proof}

\begin{lemma}  \label{enf1} Given different words $x$ and $y$ we get that
\begin{equation} \label{kw2} \int \mathfrak{a}_x\, \mathfrak{a}_x\, d \mu=0.
\end{equation}

\end{lemma}
\begin{proof} First note that it follows from orthogonality of the family of functions of the form $e_{[0x]}$ (where $x$ is a finite word) that
$$\int_{[0]} \mathfrak{a}_x\, \mathfrak{a}_x\, d \mu=\int_{[0]} \hat{\mathfrak{a}}_x\, \hat{\mathfrak{a}}_x\, d \mu=0.$$

Now, on Lemma \ref{enf2} take $f=\hat{\mathfrak{a}}_x\, \hat{\mathfrak{a}}_x$. Then, it follows that
$$\int_{[1]} \mathfrak{a}_x\, \mathfrak{a}_x\, d \mu=0.$$

As
$$ \int \mathfrak{a}_x\, \mathfrak{a}_x\, d \mu=  \int_{[0]}
\mathfrak{a}_x\, \mathfrak{a}_x\, d \mu+  \int_{[1]}
\mathfrak{a}_x\, \mathfrak{a}_x\, d \mu$$
the claim follows.

For each fixed $n \in\mathbb{N}$ we get that the support of each function
$\mathfrak{a}_x$, where $x \in \mathfrak{C}_n$, is
$[0x0] \cup[0x1] \cup [1x0] \cup [1x1]$.  When $x$ ranges in $x \in \mathfrak{C}_n$ this defines a partition of
$\mathfrak{C}_{n+2}$.

In a similar way as in the other cases we have considered before, we can produce a Fourier-like basis from the Haar basis $\mathfrak{a}_x$, where
$x$ is a finite word. Indeed, for each $n\in \mathbb{N}$ take
$$\hat{\rho}_n =  \sum_{x \in \mathfrak{C}_n } \mathfrak{a}_x.$$

From Lemma  \ref{enf1} this family is orthogonal. Dividing each element by its $L^2$ norm we can get an orthonormal family which will be also denoted by $\hat{\rho}_n$, $n \in \mathbb{N}$. Now,  collecting all the claims we proved before we get that the family $\hat{\rho}_n$ is a Fourier-like basis for the kernel of the Ruelle operator $\mu_J.$

\end{proof}

\medskip

We thank the referee of this work for the valuable suggestions for several changes on the text  and for the careful reading.

\end{document}